\documentclass[11pt]{article}
\usepackage{latexsym}
\usepackage{amsmath,amsthm,amsfonts,amssymb}
\usepackage[T1]{fontenc}
\usepackage[latin1]{inputenc}
\usepackage{aeguill}% pour les guillemets francais et autres bizrreries du latin1
\usepackage{url}%utile pour les preprints avec url
\usepackage{enumerate}% \begin{enumerate}[cas 1]
\usepackage{mdwlist} % enumerate*
\usepackage{graphicx}
\usepackage[a4paper,textwidth=15cm,textheight=25cm]{geometry}
\usepackage{export}%pour sauvecompteurs
\usepackage{hyperref}
\usepackage[nobysame,alphabetic,non-compressed-cites,non-sorted-cites]{amsrefs}

\newtheorem{thm}{Theorem}[section]

\newtheorem*{thm*}{Theorem}
\newtheorem{dfn}[thm]{Definition} 
\newtheorem*{dfn*}{Definition}

\newtheorem{cor}[thm]{Corollary}
\newtheorem*{cor*}{Corollary}

\newtheorem{prop}[thm]{Proposition} 
\newtheorem*{prop*}{Proposition} 
\newtheorem*{properties*}{Properties} 
 
\newtheorem{lem}[thm]{Lemma} 
\newtheorem*{lem*}{Lemma}

\newtheorem*{claim*}{Claim} 
 
\newtheorem*{fact*}{Fact}

\newtheorem*{qst*}{Question}

\newtheorem*{pb*}{Problem}

\theoremstyle{remark}
 
\newtheorem*{algo*}{Algorithm} 
\newtheorem*{rem*}{Remark}
\newtheorem{rem}[thm]{Remark}
\newtheorem*{example*}{Example}
\newtheorem{example}[thm]{Example}

\newenvironment{SauveCompteurs}[1]{%
\newcommand{\monparametre}{#1}
\openexport{\monparametre_sauve}
  \Export{thm}\Export{section}\Export{subsection}\Export{subsubsection}
\closeexport}{}

\newenvironment{UtiliseCompteurs}[1]{%
\newcommand{\monparametre}{#1}% utile car en fait #1 n'est pas accessible dans la deuxieme partie de l'environnement.
\openexport{\monparametre_aux}
  \Export{thm}\Export{section}\Export{subsection}\Export{subsubsection}
\closeexport
\Import{\monparametre_sauve}%
\renewcommand{\label}[1]{}%But=eviter les multiply defined. 
}{\Import{\monparametre_aux}}

\newcounter{numEnonceTmpInterne}% ce compteur sert a avoir un nom d'environnement unique dans enonce*
\newenvironment{enonce*}[1]{\theoremstyle{plain}\stepcounter{numEnonceTmpInterne}%
\def\a{enoncetmp\alph{numEnonceTmpInterne}}%
\newtheorem*{\a}{#1}\begin{\a}}{\end{\a}}

\makeatletter
\edef\@tempa#1#2{\def#1{\mathaccent\string"\noexpand\accentclass@#2 }}
\@tempa\rond{017}
\makeatother

\newcommand{\es}{\emptyset}
\renewcommand{\phi}{\varphi} 
\newcommand{\m} {^{-1}} 
\newcommand{\eps} {\varepsilon}

\newcommand {\ra} {\rightarrow}

\newcommand{\imp} {\Rightarrow}

\newcommand{\normal} {\vartriangleleft}
\renewcommand{\subsetneq}{\varsubsetneq}

\newcommand{\ie} {i.e.\ }

\newcommand {\cala} {{\mathcal {A}}}   
\newcommand {\calb} {{\mathcal {B}}}   
\newcommand {\calc} {{\mathcal {C}}}   
\newcommand {\cald} {{\mathcal {D}}}   
\newcommand {\cale} {{\mathcal {E}}}   
\newcommand {\calf} {{\mathcal {F}}}   
\newcommand {\calg} {{\mathcal {G}}}   
\newcommand {\calh} {{\mathcal {H}}}

\newcommand {\calk} {{\mathcal {K}}}

\newcommand {\calp} {{\mathcal {P}}}   
\newcommand {\calq} {{\mathcal {Q}}}   
\newcommand {\calr} {{\mathcal {R}}}   
   
\newcommand {\calt} {{\mathcal {T}}}

\newcommand {\calz} {{\mathcal {Z}}}

\newcommand {\bbN} {{\mathbb {N}}}

\newcommand {\bbQ} {{\mathbb {Q}}}   
\newcommand {\bbR} {{\mathbb {R}}}

\newcommand {\bbZ} {{\mathbb {Z}}}   

\newcommand{\grp}[1]{\langle #1 \rangle}

\newcommand{\Stab} {{\mathrm{Stab}}}
\newcommand{\Fix}{\mathop{\mathrm{Fix}}}

\newcommand{\Out} {{\mathrm{Out}}}

\newcommand{\Aut} {{\mathrm{Aut}}}

\newcommand{\Mod} {{\mathrm{Mod}}}

\newcommand{\Zmax}{\calz_{\mathrm{max}}}

\newcommand{\Inc}{\mathrm{Inc}}

\usepackage{ stmaryrd }
\newcommand{\fleche}{\mathop{\nnearrow}}

\usepackage{srcltx}
\usepackage[all]{xypic}

\newcommand{\Tcan}{T_{\mathrm{can}}}
\newcommand{\Gcan}{\Gamma_{\mathrm{can}}}
\newcommand{\Z}{{\mathbb {Z}}}

\usepackage{pdfsync}
\newcommand{\inc}{\subset}

\newcommand{\Rt}{$\R$-tree}

\newcommand {\R} {{\mathbb {R}}}

\newcommand{\Autu}{\Aut}

\newcommand{\Outu}{\Out}

\newcommand{\Tw}{\calt}%{\mathrm{Tw}}
\newcommand{\mo} {{\mathrm{Mod}}}

\newcommand{\mk}[1]{{#1}^{\mathrm{(t)}}}
\renewcommand{\PrintReviews}[1]{}
\renewcommand{\PrintDOI}[1]{}  

\begin{document}

\title{Splittings and automorphisms of relatively hyperbolic groups}
\author{Vincent Guirardel and Gilbert Levitt}
%\date{\today.\\ \small Fichier \texttt{\jobname.tex}}

\maketitle

\begin{abstract}
  We study automorphisms of a relatively hyperbolic group $G$.
When $G$ is one-ended, 
 we describe $\Out(G)$ using a preferred JSJ tree over subgroups that are virtually cyclic or parabolic.
In particular, when $G$ is toral relatively hyperbolic,
$\Out(G)$ is virtually built out of mapping class groups and subgroups of $GL_n(\bbZ)$ fixing certain basis elements.
 When more general parabolic groups are allowed, these subgroups of $GL_n(\bbZ)$ have to be replaced by McCool groups: automorphisms of parabolic groups
acting trivially (i.e.\ by conjugation) on certain subgroups.

Given a malnormal quasiconvex subgroup $P$ of a hyperbolic group $G$,
we view $G$ as hyperbolic relative to $P$ and
we apply the previous analysis to describe the group  $\Out(P\fleche G)$  of automorphisms of $P$ that extend to $G$:
it is virtually a McCool group.  If $\Out(P\fleche G)$ is infinite, then $P$ is a vertex group in a splitting of $G$.
If $P$ is torsion-free, then $\Out(P\fleche G)$ is of type VF, in particular finitely presented.

We also determine when $\Out(G)$ is infinite, for  $G$ relatively hyperbolic. The interesting case
is when $G$ is infinitely-ended and has torsion. When $G$ is hyperbolic, we show that $\Out(G)$ is infinite
if and only if $G$ splits over a maximal virtually cyclic subgroup with infinite center.
In general we show that infiniteness of $\Out(G)$ comes from the existence of a splitting with infinitely
many twists, or having a vertex group that is maximal parabolic with infinitely many automorphisms acting
trivially on incident edge groups.
\end{abstract}

 \section{Introduction}
 
 This paper studies  automorphisms of hyperbolic and relatively hyperbolic groups in relation with their splittings.
  The first  result in this direction is due to Paulin \cite{Pau_arboreal}:   if $G$ is a hyperbolic group with $\Out(G)$   infinite,
  then $G$ has an action on an $\bbR$-tree with  virtually cyclic (possibly finite) arc stabilizers. 
 Rips theory then implies that  $G$ splits over a  virtually cyclic group.

Understanding the global structure of $\Out(G)$ requires   techniques which depend on the number of ends of $G$. 
   If $G$ is one-ended, there is  an $\Out(G)$-invariant JSJ decomposition, and its study leads to   Sela's description of  $\Out(G)$ 
  as a virtual extension of a direct product of mapping class groups
  by a virtually abelian group \cite{Sela_structure,Lev_automorphisms}.  
  If $G$ has infinitely many ends, one does not get such a precise description because there is no $\Out(G)$-invariant   splitting. One may study $\Out(G)$ by letting it act on a suitable space of splittings, the most famous being Culler-Vogtmann's outer space for $\Out(F_n)$.

Before moving on to relatively hyperbolic groups, here is a basic problem about which we get new results in the context of hyperbolic, and even free, groups.
Given a finitely generated subgroup $P$  of a group $G$, 
consider the group $\Out(P \fleche G)\inc\Out(P)$ consisting of outer automorphisms of $P$ which extend to automorphisms of $G$.
 What can one say about this group ?
For instance, is it finitely generated? finitely presented? 
This question was asked by D.\ Calegari for automorphisms of free groups, and we answer it when $P$ is malnormal. 

 The answer is related to splittings through the following simple remark: 
if $P$ is a vertex group in a splitting of $G$, say $G=A*_{C_1}P*_{C_2} B$, 
then any automorphism of $P$ which acts \emph{trivially} (\ie as  conjugation by some element $p_i\in P$) 
on each incident edge group $C_i$, 
extends to $G$ 
(this is the algebraic translation of the fact that any self-homeomorphism of a closed subset $Y\subset X$  which is the identity
on the frontier of $Y$ extends by the identity to a homeomorphism of $X$).

The following theorem says that this phenomenon accounts for almost all of $\Out(P \fleche G)$.

\begin{thm}[see Corollary \ref{cor_induit}] 
 \label{calhyp}
 Let $P$ be a quasiconvex malnormal subgroup of a hyperbolic group $G$. If $\Out(P \fleche G)$ is infinite, then $P$ is a vertex group in a splitting of $G$, and the group of outer automorphisms of $P$ acting trivially 
on incident edge groups has finite index in $\Out(P \fleche G)$.
\end{thm}

We call the group of outer automorphisms of $P$ acting trivially on a family of subgroups a  \emph{McCool group of $P$}, because of McCool's paper about subgroups of $\Out(F_n)$ fixing a finite set of conjugacy classes \cite{McCool_fp}. 
In this language, 
Theorem \ref{calhyp}
says that $\Out(P\fleche G)$ is virtually a McCool group of $P$. 
It is a theme of this paper 
that 
many groups of automorphisms may be understood in terms of McCool groups, and that many results concerning the full group $\Out(G)$
also apply to McCool groups (see also \cite{GL_McCool}).

The groups considered by McCool are finitely presented \cite{McCool_fp}.
In fact, they  have a finite index subgroup with a finite classifying space \cite{CuVo_moduli}.
 In \cite{GL_McCool} we extend these results to all McCool groups of torsion-free hyperbolic groups (and more generally of toral relatively hyperbolic groups).  From this, one deduces that $\Out(P \fleche G)$ has a finite index subgroup with a finite classifying space when $G$ and $P$ are as in Theorem \ref{calhyp},  with $G$  torsion-free.

Our hypotheses for Theorem \ref{calhyp}, namely quasiconvexity  and malnormality of $P$, 
imply that $G$ is hyperbolic relative to $P$ (see \cite{Bow_relhyp}).
In fact, Theorem \ref{calhyp} is just a special case of a result describing 
$\Out(P \fleche G)$ as a virtual McCool group when $G$ is relatively hyperbolic 
and  $P$ is  a maximal parabolic subgroup   (see Theorem \ref{caleg_general} for a precise statement). 
\\

This paper also addresses the question of whether    $\Out(G)$ is finite or infinite. 
It turns out that the answer is much simpler when $G$ is torsion-free, owing to the fact that $\Out(G)$ is then infinite whenever $G$ has infinitely many ends  (see Lemma \ref{freep}). 

Things are more complicated when torsion is allowed. For instance, characterizing virtually free groups with  $\Out(G)$   infinite is a non-trivial problem which was solved by Pettet \cite{Pettet_virtually}. The following theorem gives a different characterization.
We say that a 
 subgroup of $G$ is $\Zmax$ if it is  maximal for inclusion among virtually cyclic subgroups
with infinite center.

\begin{thm} [see Theorems \ref{thm_twist_hyp}  and \ref{thm_MC_infini}] \label{twinf}
Let $G$ be a hyperbolic group. 
Then $\Out(G)$ is infinite if and only 
$G$ splits over  a $\Zmax$ subgroup $C$; in this case,   any element $c\in C$ of infinite order defines a Dehn twist which has infinite order in $\Out(G)$. 

Moreover, one may decide algorithmically whether $\Out(G)$ is finite or infinite.
\end{thm}

The first assertion
answers a question asked by D.\ Groves.
See \cite{Carette_automorphism} for a related result proved
 independently by M.\ Carette, and
  \cites{Lev_automorphisms, DG2} for the one-ended case.

Let us now consider (in)finiteness of $\Out(G)$ when $G$ is relatively hyperbolic. 
Suppose that $G$ is  hyperbolic with respect to a finite family $\calp$ of finitely generated subgroups $P_i$.
Since automorphisms of $G$ need not respect   $\calp$ (for instance, $G$ may be free and $\calp$ may consist of any finitely generated malnormal subgroup), we consider the group $\Out(G;\calp)$ consisting of automorphisms mapping each $P_i$ to a conjugate (in an arbitrary way). 
Note that $\Out(G;\calp)$ has finite index in the full group $\Out(G)$
 when the groups $P_i$ are small but not virtually cyclic, more generally when they are not themselves relatively hyperbolic in  a nontrivial way \cite{MiOs_fixed}.

Given a splitting of $G$, we have already pointed out that any automorphism of a vertex group
acting trivially on incident edge groups extends to an automorphism of $G$.
\emph{Twists} around edges of the splitting also provide automorphisms of $G$.
For instance, if $G=A*_C B$, and $a\in A$ centralizes $C$, there is an automorphism of 
$G$ equal to conjugation by $a$ on $A$ and to the identity on $B$.
Note that we do not require that $C$ be virtually cyclic or that $a\in C$
(see Subsection \ref{arb}).

The following result says that 
infiniteness of $\Out(G;\calp)$ comes from twists or from a McCool group of a parabolic group.

\begin{thm}[see Corollary \ref{cor_outu_infini}] \label{twinfr} 
Let
$G$ be hyperbolic relative to a family $\calp=\{P_1,\dots,P_n\}$ of finitely generated subgroups. 
  Then  $\Out(G;\calp)$ is infinite if and only if
$G$ has a splitting over virtually cyclic or parabolic subgroups,  with each $P_i$   contained in a conjugate of a vertex group,   such that one of the following holds:
 \begin{itemize}
\item 
  the group of  twists of the splitting is infinite;
\item
   some $P_i$ is a vertex group and there are infinitely many outer automorphisms of $P_i$ acting trivially 
on incident edge groups.
\end{itemize}
 \end{thm} 
 
 As mentioned above, one can get similar results characterizing the infiniteness of McCool groups of $G$.
 We refer to  Section \ref{outinfi}, in particular Theorem \ref{thm_marked} and Corollary \ref{cor_outu_infini}, 
 for more detailed statements. 
 \\

 Let us now discuss the techniques that we use. We assume that $G$ is hyperbolic relative to $\calp$, and we distinguish two cases according to the number of ends
(technically,  we 
consider relative one-endedness, but we will ignore this in the introduction).

When $G$ is one-ended, we use a canonical $\Out(G;\calp)$-invariant decomposition $\Gcan$ of $G$, namely  (see  Subsection \ref{cano}) the JSJ decomposition over \emph{elementary} (i.e.\ parabolic or virtually cyclic) subgroups relative to parabolic subgroups  (i.e.\  parabolic subgroups have to be contained in conjugates of vertex groups).

  One may thus generalize  the description of $\Out(G)$ given by Sela for $G$ hyperbolic, and
express $\Out(G;\calp)$ in terms of mapping class groups, McCool groups of maximal parabolic subgroups, and a group of twists $\calt$. For simplicity we restrict   to a special case here (see Section \ref{gen} for a general statement).

\begin{thm}[see Corollary \ref{limgp}]\label{thm_limgp_intro}
  Let $G$ be toral relatively hyperbolic and  one-ended.
Then some finite index subgroup $\Out^1(G)$ of $\Out(G)$ fits in an exact sequence
\begin{equation*}
1\ra \Tw \ra   \Out^1(G) \ra \prod_{i=1}^p MCG^0(\Sigma_i) \times \prod_{k=1}^m 
  GL_{r_k,n_k}(\bbZ) 
 \ra 1 
\end{equation*}
where $\Tw$ is finitely generated free abelian,    $MCG^0(\Sigma_i)$ is  the group of isotopy classes of homeomorphisms of a compact surface $\Sigma_i$     mapping each boundary component to itself in an orientation-preserving way,  and  $GL_{r,n}(\bbZ) $ 
is the group of automorphisms of $\bbZ^{r+n}$ fixing the first $n$ generators. 
\end{thm}

 More generally,   McCool groups of a one-ended   toral relatively hyperbolic group $G$
have a similar description
(see Corollary \ref{cor_MCtoral}).
  A more general statement  (without restriction on the parabolic subgroups) is given in Theorems \ref{thm_struct_m} and \ref{thm_struct_m_rel}.

We also show that the \emph{modular group} of $G$,   introduced by Sela \cite{Sela_hopf,Sela_diophantine1} 
and usually defined by considering all suitable splittings of $G$, may be seen on the single splitting $\Gcan$. We refer to Section \ref{modul} for details. 

To prove Theorem \ref{calhyp}  when $G$ is one-ended, 
one applies the previous analysis, viewing  $G$ as hyperbolic relative to $P$. Note that we use a JSJ decomposition which is relative (to $P$), and over subgroups which are not small (any subgroup of $P$ is allowed).  

 Another example of the usefulness of relative JSJ decompositions is to prove the Scott conjecture about fixed subgroups of automorphisms of free groups. The proof that we give in Section \ref{fixed}, though not really new, is simplified by the use of the   cyclic JSJ decomposition relative to the fixed subgroup. 
 
 We therefore work consistently in a relative context. We fix another family of finitely generated subgroups $\calh=\{H_1,\dots,H_{q}\}$, and we define $\Out (G;\calp,\mk\calh)$ as the  group of automorphisms mapping $P_i$ to a conjugate (in an arbitrary way) and acting trivially on $H_j$  (\ie 
as conjugation by an element $g_j$ of $G$).
 
 In order to understand the structure of the automorphism group of a one-ended  relatively hyperbolic group from its
canonical JSJ decomposition,
one needs to control automorphisms of rigid vertex groups. This is made possible by 
the following generalisation of Paulin's theorem mentioned above: 

  \begin{thm}[see Theorem \ref{thm_sci}] \label{thm_sci3}
   Let $G$ be hyperbolic relative to  $\calp=\{P_1,\dots,P_n\}$,  with $P_i$ finitely
    generated and $P_i\neq G$. 
    Let $\calh=\{H_1,\dots,H_q\}$ be another family of
    finitely generated subgroups. 

If $\Out(G;\calp,\mk\calh)$ is infinite, then $G$ splits
    over an elementary  (virtually cyclic or parabolic) subgroup relative to $\calp\cup\calh$.
  \end{thm}

Note that there is no quasiconvexity or malnormality assumption on groups in $\calh$, but the automorphisms that we consider have to act
trivially on them  (see also Remark \ref{lent}).

The theorem is proved  using the Bestvina-Paulin method (extended to relatively hyperbolic groups in \cite{BeSz_endomorphisms}) 
to get an action on an \Rt\ $T$, and then  applying Rips theory as developed in \cite{BF_stable} to get a splitting. 
There are technical difficulties in the second step because $G$ may  fail to be finitely presented (the $P_i$'s are not required to be finitely presented), 
and the action on $T$ may fail to  be stable if the $P_i$'s are not  
slender; it only satisfies a weaker property which we call \emph{hypostability}, 
and in the last section we generalize \cite{BF_stable} to hypostable actions of relatively finitely presented groups.

 Theorem \ref {thm_sci3} explains why McCool groups appear in Theorems \ref{twinfr} and \ref{thm_limgp_intro}.
Indeed, given a rigid vertex group $G_v$ in a JSJ decomposition of a one-ended group,
  Theorem \ref{thm_sci3} implies that only finitely many outer automorphisms of $G_v$ extend to automorphisms of $G$.
In turn, this implies that, after passing to a finite index subgroup, automorphisms of $G$ act
trivially on edge groups of the JSJ decomposition. See Subsection \ref{autar} 
for details.

When $G$ is not one-ended, one has to consider splittings over finite groups. 
We do not have an exact sequence  as 
in
Theorem \ref{thm_limgp_intro} 
 because there is no $ \Outu (G;\calp)$-invariant splitting. In order to prove Theorems \ref{twinf} and \ref{twinfr}, we use the tree of cylinders introduced in \cite{GL4} to obtain a non-trivial splitting over finite groups which is invariant or has an infinite group of twists (Corollary \ref{propc_induct}).

  The paper is organized as follows. Section \ref{prelim}  consists of preliminaries (JSJ decompositions, automorphisms of a tree, trees of cylinders). Section \ref{relhyp} contains generalities about relatively hyperbolic groups. We point out that vertex groups of a splitting over relatively quasiconvex subgroups are relatively quasiconvex, and that the canonical JSJ decomposition $\Gcan$ has finitely generated edge groups. In Section \ref{gen} we study the structure of the automorphism group of a one-ended relatively hyperbolic group.
  Section \ref{modul} is devoted to the modular group. In Section \ref{sec_induced} we study extendable automorphisms; Theorem \ref{calhyp} is a special case of Theorem  \ref{caleg_general}.  Section \ref{outinfi} is devoted to the question of whether groups like $\Out (G;\calp,\mk\calh)$ and $\Out (G; \mk\calp,\mk\calh)$ are finite or infinite.  Section \ref{fixed} contains a proof of the Scott conjecture, and a partial generalization to relatively hyperbolic groups. Theorem \ref{thm_sci3} is proved in Section \ref{sec_Rips}.

Acknowledgments. We thank D.\ Groves and D.\ Calegari for asking stimulating questions, and  the organizers of the 2007 Geometric Group Theory program at MSRI where this research was started. We also thank the referee for helpful comments. 
This work was partly supported by 
ANR-07-BLAN-0141, ANR-2010-BLAN-116-01,
  ANR-06-JCJC-0099-01,
ANR-11-BS01-013-01.

\tableofcontents

\section{Preliminaries}\label{prelim}

 Unless mentioned otherwise, $G$ will always be a finitely generated group.
 
 Given a group $A$ and a subgroup $B$, we denote by $Z(A)$ the center of $A$, by $Z_A(B)$ the centralizer of $B$ in $A$, and  by $N_A(B)$ the normalizer of $B$ in $A$.  
 We  write $B^g$ for $gBg\m$.
 
A subgroup $B\inc A$ is \emph{malnormal} if $B^g\cap B$ is trivial for   all $g\notin B$, \emph{almost malnormal}  if $B^g\cap B$ is finite for all $g\notin B$.
 
A group is  \emph{virtually cyclic} if it has a cyclic subgroup of finite index; it may be finite or infinite. Its outer automorphism group is finite.

 A group $G$ is \emph{slender} if  $G$ and  all its subgroups are finitely generated. We say that $G$ is \emph{small}  if it contains no non-abelian free group (see \cite{BF_bounding} for a slightly weaker  definition).
 
  Let $\calp$ be a   family of subgroups $P_i$. 
   In most cases, $\calp$ will be a finite collection of finitely generated groups $\calp=\{P_1,\dots,P_n\}$.

The group $G$ is   \emph{finitely presented relative to $\calp=\{P_1,\dots, P_n\} $} 
if it is the quotient of $P_1*\dots* P_n*F $ by the normal closure of a finite subset, with $F$ a finitely generated free group. If $G$ is finitely presented relative to $\calp=\{P_1,\dots, P_n\} $, and if $g_1,\dots,g_n\in G$,
then $G$ is also finitely presented relative to $\{P_1^{g_1},\dots,P_n^{g_n}\}$.

Let $H$ be a finitely generated subgroup. If $G$ is    finitely presented relative to $\calp$,   then $G$ is    finitely presented relative to $\calp\cup\{H\}$; the converse is not true in general.

 \subsection{Relative automorphisms}\label{sec_relaut}
 
 Given $G$ and $\calp$, we denote by 
$\Aut (G;\calp)$ the subgroup of $\Aut(G)$ consisting of automorphisms mapping each $P_i$ to a conjugate, and by  $\Out (G;\calp)$ its image in $\Out(G)$. If $\calp=\{P_1,\dots, P_n\}$, we   use the notations $\Autu (G;P_1,\dots, P_n)$ and $\Out (G;P_1,\dots, P_n)$. We also write $\Out (G; \calp, \calq)$ for $\Out (G; {\calp\cup\calq})$.

We define
$\Out (G;\mk\calp)\inc\Outu (G;\calp)$ by restricting to automorphisms whose restriction to each $P_i$ agrees with conjugation by some element   $g_i$ of $G$ (we call them \emph{marked automorphisms}, 
or  \emph{automorphisms acting trivially on $\calp$}). An outer automorphism $\Phi$ belongs to $\Out  (G;\mk\calp)$ if and only if, for each $i$, it has a representative $\alpha_i\in \Aut(G)$ equal to the identity on $P_i$. 

The group $\Out (G;\mk\calp)$ is denoted    by PMCG in \cite{Lev_automorphisms}, by  $\Out_m(G;\calp)$ in \cite{DG2}, and is called a (generalized) McCool group in \cite{GL_McCool}.

Note that $\Out(G;\mk G)$ is trivial, and that $\Out (G;\mk\calp)$ has finite index in $\Out (G;\calp)$ if $\calp$ is   a finite collection of   finite groups  (more generally, of groups $P$ with $\Out(P)$ finite).

If  $\calh=\{H_1,\dots,H_{q}\}$ 
is another family of subgroups,
we define
$\Out (G;\calp,\mk\calh)=\Out(G;\calp)\cap\Out(G;\mk\calh)$. We allow $\calp$ or $\calh$ to be empty, 
 in which case we do not write it. 

 The groups defined above do not change if we replace each $P_i$ or $H_j$ by a conjugate, 
or if we add conjugates of the $P_i$'s to $\calp$ or  conjugates of the $H_j$'s to $\calh$.

\subsection{Splittings}

A \emph{splitting} of a group $G$ is an isomorphism between $G$ and the fundamental group of a graph of groups $\Gamma$. Equivalently, using Bass-Serre theory, we view a splitting of $G$ as an action of $G$ on a simplicial tree $T$, with $T/G=\Gamma$. This tree is well-defined up to equivariant isomorphism,  
and two splittings are considered equal if there is an equivariant isomorphism between the corresponding Bass-Serre trees.

The group $\Out(G)$ acts on the set of splittings of $G$  (by changing the isomorphism between $G$ and $\pi_1(\Gamma)$, or precomposing the action on $T$).

Trees will always be simplicial trees with an action of $G$ without inversion.  We usually assume that the splitting is  \emph{minimal} (there is no proper $G$-invariant subtree). Since $G$ is  assumed to be finitely generated, this implies that $\Gamma$ is  a finite graph. 

A splitting is \emph{trivial} if $G$ fixes a point in $T$ (minimality then implies that $T$ is a point).

A splitting is \emph{relative to $\calp$} if every $P_i$   is conjugate to a subgroup of a vertex group, or equivalently if $P_i$ is  \emph{elliptic}  (i.e.\ fixes a point) in $T$. If  $\calp=\{P_1,\dots,P_{n}\}$, we also say that the splitting is relative to $P_1,\dots, P_n$.

 The group $G$ \emph{splits  over a subgroup $A$} (relative to $\calp$) if there is a non-trivial 
minimal splitting (relative to $\calp$) such that 
 $A$  is  an edge group. The group is \emph{one-ended relative to $\calp$} if it does not split over a finite group relative to $\calp$. 

If $\Gamma$ is a graph of groups, we denote by $V$ its set of vertices, and by $E$ its set of oriented edges. The origin of an edge $e\in E$ is denoted by $o(e)$. A vertex $v$ or an edge $e$ carries a group $G_v$ or $G_e$, and there is an inclusion  $i_e:G_e\to G_{o(e)}$

\subsection{Trees  and deformation spaces \cite{For_deformation,GL2}}\label{defsp}

In this subsection we consider trees rather than graphs of groups. We   denote by $G_v$ or $G_e$ the stabilizer of a vertex or an edge.

We often restrict edge stabilizers of $T$ by requiring that they belong to a family $\cala$ of subgroups of $G$ which is stable under taking subgroups and under conjugation. We then say that $T$ is an \emph{$\cala$-tree}. For instance, $\cala$ may be the set of finite subgroups, of cyclic subgroups,  of abelian subgroups, of elementary subgroups of a relatively hyperbolic group (see Section \ref{relhyp}).  We then speak of  cyclic, abelian, elementary trees (or splittings).
 
Besides restricting edge stabilizers, we also  often restrict to trees $T$ \emph{relative to $\calp$}: every $P_i$ is elliptic in $T$. We then say that $T$ is an \emph{$(\cala,\calp)$-tree}.

A tree $T'$ is a \emph{collapse} of $T$ if it is obtained from $T$ by collapsing 
each edge in a certain $G$-invariant
collection 
to a point; conversely, we say that $T$ \emph{refines} $T'$. 
In terms of graphs of groups, one passes from $\Gamma=T/G$ to $\Gamma'=T'/G$ by collapsing edges;
  for each vertex $v'\in\Gamma'$, 
the vertex group $G_{v'}$ is the fundamental group of the graph of groups occuring as the preimage of $v'$ in $\Gamma$.

Given two trees $T$ and $T'$, we say that $T$ \emph{dominates} $T'$ if there is a $G$-equivariant map $f:T\to T'$, or equivalently if every subgroup which is elliptic in $T$ is also elliptic in $T'$.  In particular, $T$ dominates any collapse $T'$.  

Two trees  belong to the same \emph{deformation space} if they dominate each other. In other words, a deformation space $\cald$ is   the set of all trees having a given family of subgroups as their elliptic subgroups. We denote by $\cald(T)$ the deformation space  containing a tree $T$, and by $\Out(\cald)\inc\Out(G)$ the group of automorphisms leaving $\cald$ invariant. The set of $\cala$-trees contained in a deformation space is called a deformation space over $\cala$ (and usually denoted by $\cald$ also).

A tree is \emph{reduced} if $G_e\ne G_v,G_w$ whenever an edge $e$ has its endpoints $v,w$ in different $G$-orbits. Equivalently, no tree obtained from $T$ by collapsing the orbit of an edge belongs to the same deformation space as $T$. If $T$ is not reduced, one may collapse edges so as to obtain a reduced tree in the same deformation space.

Any two  reduced trees in a 
deformation space over finite groups may be joined by slide moves (see  \cite{For_deformation,GL2} for definitions). 
In particular, they have the same set of edge and vertex stabilizers.

\subsection{Induced structures} \label{ind}

\begin{dfn} [Incident edge groups $\Inc_v$]  \label{iv}
Given a vertex $v$ of a graph of groups $\Gamma$, we denote by $\Inc_v$ the collection of all subgroups $i_e(G_e)$  of $G_v$, for $e$   an  edge with origin $v$. We call $\Inc_v$ the set of \emph{incident edge groups.} We also use the notation $\Inc_{G_v}$.
\end{dfn}

Similarly, if $v$ is a vertex of a  (minimal) tree, 
there are finitely many $G_v$-orbits of incident edges,
and 
 $\Inc_v$ is the family of stabilizers of some representatives of these orbits.
This is a finite collection of subgroups of $G_v$,  each  well-defined up to conjugacy.

Any splitting $\Gamma_v$ of $G_v$ relative to $\Inc_v$ extends (non-uniquely) to a splitting $\Lambda$ of $G$, whose edges are those of $\Gamma_v$ together with those of $\Gamma$;    an edge $e$ of $\Gamma$ incident to $v$ is attached to a vertex of $\Gamma_v$ whose group contains $G_e$ (up to conjugacy). We call this \emph{refining}  $\Gamma $ at  $v$ using $\Gamma_v$. One recovers $\Gamma$  from $\Lambda$ by collapsing edges of 
$\Gamma_v$.

\begin{lem} \label{mar}  Consider subgroups $H\inc K\inc G$  such that, if $g\in G$ and   $H^g\inc K$, then $ g\in K$ 
(this holds in particular when $K$ is a vertex stabilizer of a tree $T$,   and $H$ is a subgroup which fixes no edge of $T$).
\begin{itemize*}
\item
If $H'\inc K$ is conjugate to $H$ in $G$, it is conjugate to $H$ in $K$. 
\item
If $\alpha\in \Aut(G)$ leaves $K$ invariant and 
 maps $H$ to $gHg\m$, then $g\in K$. \qed
 \end{itemize*}

\end{lem}

This lemma is trivial, but very useful.
Given a vertex stabilizer $G_v$ of a tree $T$,
it  allows us to define a   family 
 $\calp_{ | G_v}$   as follows (like $\Inc_v$, it 
is a 
 finite set of subgroups of $G_v$,  each well-defined up to conjugacy).
\begin{dfn}[Induced structure  $\calp_{|G_v}$] \label{indu} 
Let $\calp=\{P_i\}$ be a collection of subgroups of $G$, and let $G_v$ be a vertex stabilizer in a
tree $T$ relative to $\calp$.
For each $i$ such that $P_i$  fixes a point in the orbit of $v$, but fixes no edge of $T$,  let  $\Tilde P_i\subset G_v$ be a conjugate of $P_i$.
When defined, $\Tilde P_i$ is unique up to conjugacy in $G_v$ by Lemma \ref{mar}.
We define $\calp_{|G_v}$ as this collection of subgroups $\Tilde P_i\subset G_v$;
 we define $\calp_{|G_v}$ similarly if $G_v$ is a vertex group of $\Gamma=T/G$. 
\end{dfn}

 \begin{rem} \label{tric}
  Given $v$ and $i$, one of the following  always holds:  $P_i$ fixes a vertex of $T$ not in the orbit of $v$, 
 or some conjugate of $P_i$ fixes $v$ and  an  edge incident to $v$, 
or $P_i$ is conjugate to a group in $\calp_{ | G_v}$.
\end{rem}

 \begin{rem} 
   In this definition, $\calp_{|G_v}$ depends not only on $\calp$ and $G_v$, but also 
on the incident edge groups of $G_v$.
In practice, we will not work with $\calp_{|G_v}$ alone, but with  $\calq_v=\calp_{|G_v}\cup\Inc_v$. This is  
the case  for instance in the following lemma.
 \end{rem}

\begin{lem} \label{extens}
If $\Gamma_v$ is a  splitting of $G_v$ relative to $\calq_v=\Inc_v\cup\calp_{ | G_v}$, 
 refining  $\Gamma$ at $v$  
 using   $\Gamma_v$   yields a   splitting $\Lambda$ of $G$ which is relative to $\calp$.  
\end{lem}

\begin{proof} The refinement is possible because $\Gamma_v$ is relative to $\Inc_v$. 
 Each $P_i$ is elliptic in  $\Lambda$ by Remark \ref{tric}.
\end{proof}

\subsection{JSJ decompositions \cite{GL3a}}\label{JSJ}

Fix  $\cala$ as in Subsection \ref{defsp}, and a (possibly empty) family  $\calp$
of subgroups. All trees considered here are $(\cala,\calp)$-trees. 

A subgroup $H$ is \emph{universally elliptic} if it is elliptic (fixes a point) in every tree. A tree is universally elliptic if its edge stabilizers are. 
A tree $T$ is a \emph{JSJ tree} over $\cala$ relative to $\calp$ if it is universally elliptic  and  dominates every universally elliptic tree    (see Section 4 of  \cite{GL3a}).  JSJ trees exist if $G$ is finitely presented relative to $\calp$.
The set of JSJ trees, if non-empty, is a deformation space  called the \emph{JSJ deformation space} over $\cala$ relative to $\calp$.
  When $\cala$ is the set of cyclic (abelian, elementary...)  subgroups, we refer to the cyclic (abelian, elementary...) JSJ.
 
When $\cala$ is the set of finite subgroups, and $\calp=\es$, 
the JSJ deformation space is the \emph{Stallings-Dunwoody  deformation space}, characterized by the property that  its trees have   vertex stabilizers with at most one end   (see Section 6 of  \cite{GL3a}).  
Because it is a deformation space over finite groups, all its reduced trees have the same edge and vertex stabilizers (see Subsection \ref{defsp}). 
This space exists if  and only if $G$ is accessible, in particular when $G$ is finitely presented  \cite{Dun_accessibility}.  
If $G$ is torsion-free, the Stallings-Dunwoody deformation space is
the \emph{Grushko deformation space}; edge stabilizers are trivial, vertex  stabilizers are freely indecomposable and non-cyclic.

More generally, the JSJ deformation space over finite subgroups relative to $\calp$ will be called the Stallings-Dunwoody deformation space relative to $\calp$. 
  We will also consider JSJ spaces over finite subgroups of cardinality bounded by some $k$; these exist whenever $G$ is finitely generated by Linnell's accessibility  \cite{Linnell}.

If $T $ is a tree (in particular, if it is a JSJ tree), a vertex stabilizer $G_v$ of $T $  not belonging to $\cala$ (or $v$ itself)  is \emph{rigid} if it is  universally elliptic.
Otherwise, $G_v$ (or $v$) is \emph{flexible}.
In many situations, flexible vertex stabilizers  $G_v$ of JSJ trees are  {quadratically hanging} subgroups   (see Section 7 of  \cite{GL3a}). 

\begin{dfn} [QH vertex] \label{dqh}
 A vertex stabilizer  $G_v$ (or $v$) is \emph{quadratically hanging},   or \emph{QH},
  (relative to $\calp$)  if 
there is a normal subgroup $F\normal G_v$ (called the \emph{fiber} of $G_v$)
such that $G_v/F$ is isomorphic to the fundamental group  $\pi_1(\Sigma)$ of a hyperbolic $2$-orbifold $\Sigma$ (usually with boundary); moreover, if $H\inc G_v$ is an incident edge stabilizer, or is the intersection of $G_v$ with a conjugate of a    group in $\calp$, then the image of $H$ 
 in $\pi_1(\Sigma)$ is   finite or contained in a  \emph{boundary subgroup}  
(a subgroup conjugate to the fundamental group of a boundary component).  
\end{dfn}

\begin{dfn}[Full boundary subgroups] 
\label{bv}  Let $G_v$ be a QH vertex stabilizer. For each boundary component of $\Sigma$, we select a    representative for the conjugacy class of its fundamental group in $\pi_1(\Sigma)$,
  and we consider  its full preimage in $G_v$.
This defines a finite family $\calb_v$ of subgroups of $G_v$. 
\end{dfn}
 If  $G_v$ is QH with finite fiber,   
 every infinite  incident edge stabilizer  
is virtually cyclic and (up to conjugacy in $G_v$) contained with finite index in a group of $\calb_v$. 
  If $G$ is one-ended relative to $\calp$, then   every incident edge stabilizer  is infinite, so 
  is contained in a group of $\calb_v$.

\begin{rem}\label{rem_UEQH}
 If $G_v$ is a \emph{flexible} QH vertex stabilizer
  with finite fiber, and $H\inc G_v$ is universally elliptic, 
then the image of $  H$  in $\pi_1(\Sigma)$ is finite or contained in a boundary subgroup (see Proposition 7.6 of \cite{GL3a}; this requires a technical assumption on $\cala$, which holds in all cases considered in the present paper).
  In particular, $H$ is virtually cyclic. 
\end{rem}

\subsection{The automorphism group of a tree   \cite{Lev_automorphisms}}\label{arb}

Let $T$ be a tree with a minimal  action of $G$.   We   assume that $T$ is not a line with $G$ acting by translations. 

We denote by $\Aut(T)\inc\Aut(G)$ the group 
of automorphisms $\alpha$  \emph{leaving $T$ invariant}: 
  there exists an  isomorphism  
$f_\alpha:T\ra T$ which is   $\alpha$-equivariant  in the sense that $f_\alpha(gx)=\alpha(g)f_\alpha(x)$ for $g\in G$ and $x\in T$.

Following \cite{Lev_automorphisms}, 
we describe
the image $\Out(T)$ of $\Aut(T)$ in $\Out(G)$   in terms of the graph of groups $\Gamma=T/G$.   Our assumptions on $T$ imply  that $\Gamma$ is minimal, and is not a mapping torus (as defined in \cite{Lev_automorphisms}).

The group $\Out(T)$ acts on   the finite graph $\Gamma=T/G$, 
 and we define $\Out^0(T)$ as the finite index subgroup acting trivially. 
We use the notations $\Out(T;\calp,\mk\calh)=\Out(T)\cap\Out(G;\calp,\mk\calh)$,
and $\Out^0(T;\calp,\mk\calh)=\Out^0(T)\cap\Out(G;\calp,\mk\calh)$   (see Section \ref{sec_relaut}).

\paragraph{Action on vertex groups.}

 If $v\in V$ is a vertex of $\Gamma$, there is a  natural map $\rho_v:\Out^0(T)\to  \Out(G _v)$ defined as follows.  Let $\Phi\in\Out^0(T)$. When $N_G(G_v) $ acts on $G_v$ by inner automorphisms, in particular when $G_v$ fixes a unique point in $T$ (in this case  $N_G(G_v)=G_v$), one defines $\rho_v(\Phi)$ simply by choosing any representative  $\alpha\in\Aut(T)$   of $\Phi$ leaving $G_v$ invariant and   considering its restriction to $G_v$. In general, one has to choose $\alpha$ more carefully:  one  fixes  a vertex $\tilde v$ of $T$ mapping to $v$   such that the stabilizer of $\Tilde v$ is $G_v$, one chooses $\alpha$ so that $f_\alpha$ fixes $\tilde v$, and 
$\rho_v(\Phi)$ is represented by  the restriction $\alpha_{|G_{v}}$. 

 If $e$ is an edge of $\Gamma$, one may define $\rho_e:\Out^0(T)\to  \Out(G _e)$ similarly.

 Let $\rho :\Out^0(T)\to \prod_{v\in V}\Out(G _v)$ be the product map. 
 As   pointed out 
 in    Subsection 2.3 of \cite{Lev_automorphisms}, the image $\rho (\Out^0(T))$ contains $\prod_{v\in V}\Out  (G_v;\mk\Inc_v)$ 
and is contained in $\prod_{v\in V}\Out (G_v;\Inc_v)$.  More precisely:
 
\begin{lem}\label{automind} Let $\calp,\calh$ be two  
families of  
subgroups of $G$.
  Let $T$ be a tree relative to $\calp\cup \calh$. 
Then  
  \begin{equation}\label{eq_incl}
\Out(G_v;\mk\Inc_v,\calp_{|G_v},\mk\calh_{|G_v})
\subset \rho_v (\Out^0(T;\calp,\mk \calh))
\subset  
\Out(G_v;\Inc_v,\calp_{|G_v},\mk\calh_{|G_v})  \end{equation}
for every $v\in V$, and 
 $$\prod_{v\in V} 
\Out(G_v;\mk\Inc_v,\calp_{|G_v},\mk\calh_{|G_v})
\subset \rho (\Out^0(T;\calp,\mk \calh))
\subset  \prod_{v\in V}
\Out(G_v;\Inc_v,\calp_{|G_v},\mk\calh_{|G_v}).$$

 If $\Out(G_e)$ is  finite for all edges   (resp.\  for all edges $e$ incident to $v$),    
all  inclusions (resp.\ all inclusions  in   (\ref{eq_incl})) have  images of finite index.
\end{lem}

Recall that $\calp_{|G_v}$ was defined in Definition \ref{indu}.

\begin{proof}
The   inclusion $ \prod_{v\in V}\Out(G_v;\mk\Inc_v)\subset \rho (\Out^0(T))$ is proved in \cite{Lev_automorphisms} by extending any
  $\Phi_v\in \Out (G_v;\mk\Inc_v)$ 
``by the identity''  to get  
  $\Phi\in\Out^0(T)$,  with $\Phi_v=\rho_v(\Phi)$, acting as a conjugation on each edge group 
and on each $G_w$ for $w\ne v$.   The     left hand side inclusions  in the lemma
follow from Remark \ref{tric}.

The inclusion $\rho_v(\Out^0(T))\subset \Out(G_v;\Inc_v)$
follows from the fact  that, given an edge $e$ of $T$ containing the lift $\tilde v$ of  $v$ used to define $\rho_v$,
 any $\Phi\in \Out^0(T)$ has a representative  $\alpha$ such that   $f_\alpha$ fixes $e$; this representative induces an automorphism of $G_{\tilde v}$ leaving $G_e$ invariant. To prove the right hand side inclusions, apply Lemma \ref{mar} with $K=G_{\tilde v}$, recalling that groups in    
$\calh_{|G_v}$ or $\calp_{|G_v}$ fix a unique point in $T$.

 If  $\Out(G_e)$ is finite for all incident edge groups,
  $\Out(G_v;\mk\Inc_v)$ has finite index in $\Out(G_v;\Inc_v)$ (see Proposition 2.3 in \cite{Lev_automorphisms}).
Intersecting with $\Out(G_v;\calp_{|G_v},\mk\calh_{|G_v})$,
we get that $\Out(G_v;\mk\Inc_v,\calp_{|G_v},\mk\calh_{|G_v})$
 has finite index in
$ \Out(G_v;\Inc_v,\calp_{|G_v},\mk\calh_{|G_v})$.
This conludes the proof.
\end{proof}

 \begin{rem}\label{automind2} 
We can also consider automorphisms which do not leave $T$ invariant,  but 
 only leave some vertex stabilizer $G_v$ invariant. 
Assuming that $G_v$ equals its normalizer, there is  a natural map $\rho_v:\Out(G;G_v)\to\Out(G_v)$ and 
$$ 
\Out(G_v;\mk\Inc_v,\calp_{|G_v},\mk\calh_{|G_v})
\subset \rho_v (\Out(G;G_v,\calp,\mk \calh))
\subset  
\Out(G_v; \calp_{|G_v},\mk\calh_{|G_v}).$$

\end{rem}

\paragraph{Twists.}

  As in Subsection 2.5 of  \cite{Lev_automorphisms},
 we now consider the kernel of the  product map
$\rho :\Out^0(T)\to \prod_{v\in V}\Out(G _v)$. 
It consists of  automorphisms in $\Out^0(T)$ having, for each $v$,  
a representative in  $\Aut(T)$ whose restriction to  $G_v$ is the identity.
 If $T$ is relative to  $\calp$,  the group $\ker\rho$ is contained in 
$\Out(T;\mk\calp
)
$.

To study  $\ker\rho$, we need to introduce the \emph{group of twists} $\calt$ associated to $T$ or equivalently to $\Gamma$ (we write $\Tw(T)$ or $\Tw(\Gamma)$ if there is a risk of confusion). 

Let $e $ be 
a separating edge of $\Gamma$ with origin $v$ and endpoint $w$. Then $G=A*_{G_e}B$ with $G_v\inc A$ and $G_w\inc B$. Given $g\in Z_{G_v}(G_e)$, one defines the  
\emph{twist by $g$ around $e$ near $v$} as the (image in $\Out(G)$ of the) automorphism of $G$ equal to the identity on $B$ and to conjugation by $g$ on $A$. There is a similar definition   in the case of an HNN extension $G=A*_C=\grp{ A,t\mid tct\m=\phi(c),c
\in C}$: given $g\in Z_A(C)$,
  the twist by $g$ is the identity on $A$ and   sends $t$ to $tg$.

The group  $\calt$ is the subgroup of $\Out(G)$ generated by all twists. 
It is a quotient of $\prod_{e\in E}Z_{G_{o(e)}}(G_e)$ and is contained in $\ker\rho$. 
The following facts follow directly from Section 2 of \cite{Lev_automorphisms}.

\begin{lem}  \label{lem_virt_nobitwist}
\begin{enumerate}
\item
If every $\Out(G_e)$ is finite, then $\calt$ has finite index in $\ker \rho$.
\item Assume that every non-oriented
edge $e$ of $\Gamma$ has an endpoint $v$ such that $N_{G_v}(G_e)$ acts on $G_e$ by inner automorphisms (this holds in particular if $G_v$ is abelian, or if $G_e$ is malnormal in $G_v$, or if $G_e$ is infinite and almost malnormal  in $G_v$). Then $\calt=\ker\rho$. \qed
\end{enumerate}
\end{lem}

The kernel of the epimorphism from $\prod_{e\in E}Z_{G_{o(e)}}(G_e)$ to $\calt$ is the image of a natural map $$j:\prod_{v\in V}Z(G_v)\times\prod_{e\in\cale}Z(G_e)\to\prod_{e\in E}Z_{G_{o(e)}}(G_e)$$
where $\cale$ is the set of non-oriented edges of $\Gamma$   (see Proposition 3.1 of  \cite{Lev_automorphisms}). The image of an element of $Z(G_v)$ is called a \emph{vertex relation} at $v$, the image of an element of $Z(G_e)$ is   an \emph{edge   relation}.

 For instance, if $\Gamma$ is a  non-trivial amalgam $G=A*_C B$, then $\Tw$ is the image of the map $p:Z_A(C)\times Z_B(C)\to\Out (G)$ sending $(a,b)$ to the class of the automorphism acting on $A$ as conjugation by $a$ and on $B$ as conjugation by $b$. The kernel of $p$ is generated by the elements      $(a,1)$ with $a\in Z(A)$ and $(1,b)$ with $b\in Z(B)$ (vertex relations), together with the elements
  $(c,c)$ with $c\in Z(C)$ (edge relations). 

\begin{lem}\label{lem_twist0}
Let $e$ be an edge of $\Gamma$ with origin $v$. If $Z(G_e)$   and $Z(G_v)$ are finite, but
  $Z_{G_v}(G_e)$ is  infinite, 
then  the image of  $Z_{G_v}(G_e)$ in $\Tw $  is  infinite. 
In particular, $\Tw $ is infinite. 
\end{lem}

Note that $Z_{G_v}(G_e)$ is infinite if $N_{G_v}(G_e)$ is infinite and $G_e$ is finite.

\begin{proof}
It is pointed out in \cite[Lemma 3.2]{Lev_GBS} that  the image of  $Z_{G_v}(G_e)$ in $\Tw$ maps onto the quotient $Z_{G_v}(G_e)/\langle Z(G_v), Z(G_e)\rangle$. Since $Z(G_v)$ and $Z(G_e)$ are commuting finite subgroups, this quotient is infinite.
\end{proof}

Let $\Gamma$ be a graph of groups with fundamental group $G$, and $\Gamma_0\subset\Gamma$ a connected subgraph. We view $\Gamma_0$ as a graph of groups, with fundamental group $G_0\inc G$ and     associated group of twists $\Tw(\Gamma_0)\inc\Out(G_0)$.

\begin{lem}\label{lem_extension_twist}
If $\Tw (\Gamma_0)$ is infinite, then so is $\Tw(\Gamma)$.
\end{lem}

\begin{proof}  
Let $E_0$ be the set of oriented edges of $\Gamma_0$. The projection from $\prod_{e\in E}Z_{G_{o(e)}}(G_e)$ to $\prod_{e\in E_0}Z_{G_{o(e)}}(G_e)$ obtained by keeping only the factors  with $e\inc \Gamma_0$ is compatible with the vertex and edge relations, so induces 
an epimorphism from $\Tw (\Gamma)$ to $\Tw(\Gamma_0)$.
\end{proof}

\begin{lem}\label{lem_collapse_twist} 
If  $\Gamma$ is a graph of groups with $\Tw (\Gamma)$ infinite, there is an edge $e$ such that the graph of groups $\Gamma_e$ obtained by collapsing every edge except $e$ has  $\Tw (\Gamma_e)$ infinite. 
\end{lem}

\begin{proof}  
There is an edge $e$ such that $Z_{G_{o(e)}}(G_e)$ has infinite image in $\Out(G)$. Twists of $\Gamma$ around $e$ are also twists of $\Gamma_e$.
\end{proof}

\subsection{Trees of cylinders  \cite{GL4}}\label{defcyl}

We   recall   some basic properties of the tree of cylinders (see    Section 4 of \cite{GL4} for details).
Besides $\cala$ (and possibly $\calp$), 
we have to fix  
 a conjugacy-invariant subfamily $\cale\inc\cala$ and an admissible equivalence relation $\sim$ on $\cale$. Rather than giving  a general definition, we describe the examples that will be used in this paper   (with $\cala$ consisting of all finite, elementary, or abelian subgroups respectively):

 \begin{enumerate}
 \item $\cale$ consists of all finite subgroups of a fixed order $k$, and $\sim$ is equality.
 \item $G$ is relatively hyperbolic (see Section \ref{relhyp}), $\cale$ consists of all infinite elementary subgroups (parabolic or loxodromic), and $\sim$ is co-elementarity: $A\sim B$ if and only if $\langle A,B\rangle$ is elementary.
 \item $G$ is a torsion-free CSA group, $\cale$ consists of all infinite abelian subgroups, and $\sim$ is commutation: $A\sim B$ if and only if $\langle A,B\rangle$ is abelian (recall that $G$ is CSA if centralizers of non-trivial elements are abelian and malnormal).
 \end{enumerate}

 Let $T$ be a tree with edge stabilizers in $\cale$.  We declare two (non-oriented) edges  $e,f$ to be equivalent if $G_e\sim G_f$. 
The union of all edges in  an equivalence class is a subtree $Y$, called a  \emph{cylinder} of $T$. Two distinct cylinders meet in at most one point.
The  \emph{tree of cylinders} $T_c$ of $T$ is the bipartite tree such that   $V_0(T_c)$ is the set of vertices $x$ of $T$ which belong to at least two cylinders, 
$V_1(T_c)$ is the set  of cylinders $Y$ of $T$, and there is   an edge $\varepsilon=(x,Y)$ between $x$ and $Y$ in $T_c$ if and only if $x\in Y$. In other words, one obtains $T_c$ from $T$ by replacing each cylinder $Y$ by the cone on its boundary (defined as the set of vertices of $Y$ belonging to some  other cylinder).
  Note that $T_c$  may be trivial even if $T$ is not.

 The tree $T_c$  is dominated by $T$ (in particular, it is relative to $\calp$ if $T$ is). It  only depends on the deformation space $\cald$ containing $T$ (we sometimes say that it is the tree of cylinders of $\cald$).  In particular, $T_c$ is  invariant under   any automorphism of $G$ leaving $\cald$ invariant. 
 
The stabilizer of a vertex $x\in V_0(T_c)$ is the stabilizer of $x$, viewed as a vertex of $T$. 
The stabilizer $G_Y$ of a vertex $Y\in V_1(T_c)$ is the stabilizer $G_\calc$ of the equivalence class 
$\calc\in\cale/\sim$ containing   stabilizers of edges in $Y$,  for the action of $G$ on $\cale$ by conjugation   (see Subsection 5.1 of \cite{GL4}). It is the normalizer of a finite subgroup in case 1, a maximal elementary (resp.\ abelian) subgroup in case 2 (resp.\ 3). Note that $A\inc G_\calc$ if 
  $A\in\cale$ and $\calc$ is its equivalence class.

The stabilizer of an edge $\eps=(x,Y)$ of $T_c$ is  $G_\varepsilon=G_x\cap G_Y$; it is elliptic in $T$. In cases 2 and 3, $G_\varepsilon$ belongs to $\cala$. 
 But, in case 1, it may happen that  edge stabilizers  of $T_c$ are not in $\cala$, so
we also consider  
 the \emph{collapsed tree of cylinders} $T_c^*$ 
obtained from $T_c$ by collapsing   each edge  whose stabilizer does not belong to $\cala$ (see Subsection 5.2 of \cite{GL4}). 
It is an $(\cala,\calp)$-tree  if $T$ is, and $(T_c^*)_c^*=T_c^*$.

\section{Relatively hyperbolic groups}\label{relhyp}

In this  section we assume that $G$ is hyperbolic relative to a finite family  $\calp=\{P_1,\dots,P_n\}$ of finitely generated subgroups;  we say that $\calp$ is the \emph{parabolic structure}. 

The group $G$   is finitely generated. It is not necessarily finitely presented, but it is finitely presented relative to $\calp$   \cite{Osin_relatively},  so JSJ decompositions relative to $\calp$ exist. In particular,  there is a  deformation space of relative Stallings-Dunwoody decompositions (see Subsection \ref{JSJ}).

\subsection{Generalities}\label{gene}
We refer to \cite{Hruska_quasiconvexity} for equivalent definitions of relative hyperbolicity. In particular, $G$ acts properly discontinuously on a proper  geodesic $\delta$-hyperbolic space $X$ 
(which may be taken to be a graph \cite{GroMan_dehn}), the action is cocompact in the complement of a $G$-invariant union $\calb$ of disjoint horoballs, and the $P_i$'s are representatives of conjugacy classes of stabilizers of horoballs.  
  Any horoball $B\in\calb$ has a unique point at infinity   $\xi$, and the stabilizer of $\xi$   (for the action of $G$ on $\partial X$)   coincides with the stabilizer of $B$.

  For each constant $M>0$, one can change the system of horoballs so that any two distinct horoballs are at distance at least $M$.
Indeed, for each horoball $B$ with stabilizer $P$ defined by a horofunction $h$, the function $h'(x)=\sup_{g\in P} h(gx)$
is another (well-defined) horofunction which is $P$-equivariant; then $B'=h'{}\m([R,\infty))$ is a new $P$-invariant horoball such that   $d(B',X\setminus B)\geq M$ for $R$ large enough. 
Doing this for a chosen horoball in each orbit, and extending by equivariance, one gets a system of horoballs
at distance at least $M$ from each other.

A subgroup of $G$ is \emph{parabolic} if it is contained in a conjugate of some $P_i$, \emph{loxodromic}  if it is infinite, virtually cyclic, and not parabolic,  \emph{elementary} if it is parabolic  or virtually cyclic (finite  or loxodromic).   Any small subgroup is elementary.
The group $G$ itself is elementary if it is virtually cyclic or equal to a $P_i$.  We say that $A$ is an elementary subgroup of $B$ if it is elementary and contained in $B$.

 One may remove any virtually cyclic  
  subgroup  from $\calp$, without destroying relative hyperbolicity (see e.g.\  \cite[Cor 1.14]{DrSa_tree-graded}). Conversely, one may add to $\calp$    a finite subgroup or a maximal loxodromic  subgroup (see e.g.\ \cite{Osin_elementary}). These operations  do  not change the set of elementary (or relatively quasiconvex, as defined below)  subgroups, and it is sometimes convenient  (as in \cite{Hruska_quasiconvexity}) to assume that every $P_i$ is infinite.   Any infinite $P_i$ is a maximal elementary subgroup.

The following lemma is folklore,  but we have not found it in the literature.

 \begin{lem}\label{borne} 
 Given    a relatively hyperbolic group $G$, there exists a number  $M$ such that any 
  elementary subgroup $H\inc G$ of cardinality $>M$ is contained in a unique maximal elementary subgroup $E(H)$. There are finitely many conjugacy classes of non-parabolic finite subgroups.  
  \end{lem}

  \begin{proof}  We may assume that every $P_i$ is infinite. 
   Let $H$ be elementary.
  The existence of $E(H)$  is well-known if $H$ is infinite (see for instance \cite{Osin_relatively}), 
so assume  $H$ is finite. 
We also assume that the distance between any two distinct horoballs in $\calb$ is bigger than $6\delta$. 
Given $r>0$, let $X_r$ be the set of points of $X$ which are moved less than $r$ by $H$.
It follows from Lemma 3.3 p.\  460 of \cite{BH_metric} (existence of quasi-centres) that 
$X_{5\delta}$ is nonempty. 

Arguing as in the proof of Lemma 1.15 page 407 and Lemma 3.3 page 428 in \cite{BH_metric}, 
one sees that any geodesic joining two points of $X_{5\delta}$, or a point of  $X_{5\delta}$ to a fixed point of $H$ in $\partial X$,
 is contained in $X_{9\delta}$.

   If $X_{9\delta}$ meets $X\setminus \calb$,
    properness of the action of $G$ on $X\setminus \calb$ implies that there are only finitely many possibilities for $H$ up to conjugacy, so we can choose $M $ to ensure $X_{9\delta}\inc  \calb$.
This implies that $X_{5\delta}$ is contained in a unique horoball $B_0$ of $\calb$. This horoball is $H$-invariant since horoballs are $6\delta$-apart, so $H$ fixes the point at infinity   $\xi_0 $ of $B_0$ and is contained in the maximal parabolic subgroup $E(H)=\Stab (B_0)$.  In particular, $H$ is parabolic.  

There remains to prove  uniqueness of $E(H)$.  It suffices to check that $H$ cannot fix any point $\xi \ne \xi _0$ in $\partial X$. If it did, a geodesic joining $\xi $ to a point of $X_{5\delta}$ would be contained in $X_{9\delta}$ and meet $X\setminus \calb$. This contradicts our choice of $M$. 
  \end{proof}

 Since maximal elementary subgroups are equal to their normalizer, we get:
\begin{cor}[{\cite[Lemma 4.20]{DrSa_groups}}]\label{almar}
 Maximal elementary subgroups  $E$ are \emph{uniformly almost malnormal}: if $E\cap gEg\m$ has cardinality $>M$, then $g\in E$.  \qed
\end{cor}

\subsection{Quasiconvexity}\label{qconv} 

\begin{dfn}[Hruska \cite{Hruska_quasiconvexity}] Let $X$ and $\calb$ be as above, and $C>0$. 
A subspace $Y\inc X$ is relatively $C$-quasiconvex if, given $y,y'\in Y$, any geodesic $[y,y']\subset X$ 
has the property that $[y,y']\setminus \calb$ lies in the $C$-neighbourhood of $Y$. The space $Y$ is relatively quasiconvex if it is relatively $C$-quasiconvex for some $C$.
  A subgroup $H<G$ is relatively quasiconvex if some (equivalently, every)  $H$-orbit is relatively quasiconvex in $X$.
\end{dfn}

\begin{prop}\label{gvqc}
  Let $G$ be hyperbolic relative to $\calp=\{P_1,\dots, P_n\}$,  with $P_i$ finitely generated.
If $G$ acts on a simplicial tree $T$ relative to $\calp$ 
with relatively  quasiconvex edge stabilizers, then vertex stabilizers are relatively  quasiconvex.
\end{prop}

The proposition applies in particular if edge stabilizers are elementary, since elementary subgroups are relatively quasiconvex. It was proved by Bowditch  
\cite[Proposition 1.2]{Bo_cut} and Kapovich \cite[Lemma 3.5]{Kapovich_quasiconvexity} for $G$ hyperbolic. We generalize Bowditch's argument.

\begin{proof}
 As usual, we assume that $G$ acts on $T$  minimally and without    inversion.
Since $G$ is finitely generated,
the graph $T/G$ is   finite.  We also assume that $X$ is a  connected graph, 
 and  edges of $T$ have length 1.

Since $P_i$ is elliptic in $T$, and stabilizers of points of $X$ are finite, hence elliptic in $T$,
there exists an equivariant map $f:X\ra T$  sending vertices to vertices, mapping  each edge linearly to an edge path, 
and 
 constant on each horoball of $\calb$.

For each edge $e$ of $T$, let $m_e$ be the midpoint of $e$, and $Q_e=f\m(m_e)$.
Let $v$ be a vertex of $T$,
and let $E_v$ be the set of edges of $T$ with origin $v$.
Let $Q_v\subset X$ be the preimage under $f$ of the closed ball of radius $\frac12$ around $v$ in $T$.
Note that $Q_e\subset Q_v$ for all $e\in E_v$ and $Q_e\cap\calb=\es$.
  Also note that  $Q_e\neq\es$ by minimality of $T$ and connectedness of $X$.

 If $f(x)=f(hx)$ for   $x\in X$ and $h\in G$, then $h$ fixes $ f(x)$. Since $G
$ acts cocompactly
on $X\setminus\calb$, it follows that $G_v$ acts cocompactly on    $Q_v\setminus \calb $ and $G_e$ acts cocompactly on $Q_e=Q_e\setminus \calb $.    
In particular, $Q_e$ is the $G_e$-orbit of a finite set. 
Relative quasiconvexity of $G_e$ implies that $Q_e$ is relatively quasiconvex. 
Since $T/G$  
 is  a finite graph, there exists a common constant $C$ such that all subsets $Q_e$ are relatively $C$-quasiconvex.

We now fix a vertex $v$, and we show that $G_v$ is relatively quasiconvex. Choose $x\in Q_v\setminus\calb$. 
Since $G_v$ acts cocompactly on    $Q_v\setminus \calb $,  the Hausdorff distance between $Q_v\setminus \calb$ and the $G_v$-orbit of $x$ is finite, so it suffices to prove that $Q_v$ is relatively quasiconvex. Let $\gamma$ be a geodesic of $X$ joining two points of $Q_v$, and let $\gamma_0$ be a maximal subgeodesic contained in $\gamma\setminus Q_v$.  Considering the image of $\gamma$ in $T$, we see that 
 both  endpoints of $\gamma_0$  belong to the same $Q_e$, for some    $e\in E_v$. Thus  $\gamma_0\setminus\calb$ is $C$-close to $Q_e$, hence to $Q_v$. This shows   that $Q_v$ is relatively $C$-quasiconvex.
\end{proof}

A relatively quasiconvex subgroup is   relatively hyperbolic in a natural way (\cite[Theorem 9.1]{Hruska_quasiconvexity}). 
In particular:
\begin{lem} \label{Hrufi}
If $ G_v$ is an infinite  vertex stabilizer of a tree with finite edge stabilizers, it is   hyperbolic relative to the family $\calp_{ | G_v}$ of Definition \ref{indu}.
\end{lem}

\begin{proof} This follows from Theorem 9.1 of \cite{Hruska_quasiconvexity}, adding finite groups belonging to   $\calp_{ | G_v}$   to the 
parabolic structure if needed.
\end{proof}

\subsection{The canonical JSJ decomposition}\label{cano}

  In this section we recall the description of  the  canonical relative JSJ decomposition.
  The content of the word canonical is that the JSJ tree  (not just the JSJ deformation space) is invariant under automorphisms.

 Let  $G$ be  hyperbolic relative to $\calp$, and denote by $\cala$ the family of elementary subgroups. 
In this subsection we fix another (possibly empty) family of finitely generated subgroups  $\calh=\{H_1,\dots, H_q\}$ 
and we assume that $G$ is one-ended relative to $\calp\cup\calh$.

We consider the canonical   $\Out(G;\calp\cup\calh)$-invariant  JSJ tree $\Tcan$ defined   (under the name $T_c^*$) in Theorem 13.1 of \cite{GL3b} (see also   Theorem 7.5 of \cite{GL4}).  
It   is the tree of cylinders  (see Subsection \ref{defcyl}) of  the JSJ deformation space $\cald$ over elementary subgroups relative to $\calp\cup\calh$, and it belongs to $\cald$. It is $\Out(G;\calp\cup\calh)$-invariant  because 
  the JSJ deformation space $\cald$ is.  

  Being a tree of cylinders,  $\Tcan$ is bipartite, with   vertices 
$x\in V_0(\Tcan)$  and $Y\in V_1(\Tcan)$.   
 Stabilizers of vertices in $V_0(\Tcan)$ are non-elementary, and stabilizers of vertices in $V_1(\Tcan)$ are maximal elementary subgroups.   Non-elementary vertex stabilizers may be rigid or flexible (see Subsection \ref{JSJ}), and flexible vertex stabilizers are QH with finite fiber (see Theorem 13.1 of \cite{GL3b}). Elementary vertex stabilizers   are infinite by one-endedness, they may be parabolic or loxodromic. Thus there are exactly  four  possibilities for a   vertex $v\in \Tcan$: 
\begin{enumerate}
\item[0.a.]  \emph{rigid}: $G_v$ is non-elementary and is elliptic in every $(\cala, \calp\cup\calh)$-tree.
\item[0.b.]  \emph{(flexible) QH}: $G_v$ is  non-elementary and  not universally elliptic. Then $v$ is a flexible QH vertex with finite fiber 
 as in Subsection \ref{JSJ}.   
 \item[1.a.]   \emph{maximal parabolic}: $G_v$ is conjugate to a $P_i$. 
\item[1.b.]   \emph{maximal loxodromic}:  $G_v$ is a maximal virtually cyclic subgroup of $G$, and $G_v$ is not parabolic. 

\end{enumerate}

\begin{rem} 
A QH vertex $v\in V_0(\Tcan)$ is flexible, except in a
few   cases; for instance, if $G$ is torsion-free, the only exceptional case is when the underlying surface is a pair of pants (thrice punctured sphere). In these cases we view $v$ as rigid rather than QH. This should not cause confusion. In particular, Propositions \ref {imro} and \ref{imro_rel} would remain valid with $v$ viewed as QH. 
\end{rem}

 If $\eps=(x,Y)$ is an edge, then  $G_\eps=G_x\cap G_Y$ is an infinite  maximal elementary subgroup of   $G_x$  (but $G_\varepsilon$ may fail to be maximal  elementary  in $G$ and $G_Y$). 
In particular, $G_\varepsilon$ is always almost malnormal in $G_x$, so that Assertion 2 of Lemma \ref{lem_virt_nobitwist} applies to $\Tcan$,  showing  $\calt=\ker\rho$.

 Let now $v$ be a QH vertex. We claim that, \emph{if $\calh=\es$ and every $P_i$ is infinite, then $\calb_v=\Inc_v\cup \calp_{|G_v} $, and $\calb_v=\Inc_v$ if   no $P_i$ is virtually cyclic}  (see  Definitions \ref {iv}, \ref{indu} and \ref{bv}).

Groups in $\Inc_v\cup \calp_{|G_v} $ are infinite maximal elementary subgroups of $G_v$, and groups in  $\calb_v$ are virtually cyclic, so  $\Inc_v\cup \calp_{|G_v}\inc  \calb_v $ by Definition \ref{dqh}.  Conversely, because of one-endedness,
every boundary component  of $\Sigma$ is \emph{used} by $\Inc_v\cup \calp_{|G_v} $  \cite[Subsection 2.5 and Theorem 13.1]{GL3b}: 
for any  full boundary subgroup $B\in\calb_v$, there exists a subgroup $H\in \Inc_v\cup \calp_{|G_v} $ 
such that some $G_v$-conjugate of $H$  is a finite index subgroup of  $B$ (hence equals $B$). This proves the converse. Since groups in $\calb_v $ are virtually cyclic, $\calb_v=\Inc_v$ if   no $P_i$ is virtually cyclic. This proves the claim.

When $\calh\ne\es$, we still have $\Inc_v\cup \calp_{|G_v}\inc  \calb_v $. If $H_j$ is infinite, the intersection of any of its conjugates with $G_x$   
is contained   in  a full boundary subgroup, 
in particular  is virtually cyclic.  Conversely, a group of $\calb_v$ belongs to $\Inc_v\cup \calp_{|G_v} $
or contains with finite index a $G_v$-conjugate of a group $H\in \calh_{|G_v} $.

  This analysis implies that
a group $P_i $ which is not virtually cyclic is contained in a rigid $G_x$ 
 or is equal to some $G_Y$ (which may be contained in a rigid $G_x$). A group $H_j$ which is not virtually cyclic is contained in a rigid $G_x$ or in  a $G_Y$.

\begin{lem}\label{fingen}
$\Tcan$ has finitely generated edge (hence vertex) stabilizers.
\end{lem}

We do not assume that the $P_i$'s are slender, so  there may exist infinitely generated elementary subgroups.

\begin{proof}
Since $G$ is finitely presented relative to  $\calp\cup\calh$, 
there is an elementary  JSJ tree $T_J$  
relative to $\calp\cup\calh$ having finitely generated edge stabilizers (\cite{GL3a}, Theorem 5.1). 
The tree $\Tcan$ is the tree of cylinders of $T_J$.

Consider an edge $\eps=(x,Y)$  of $\Tcan$, and $G_\eps=G_x\cap G_Y$.  We view $Y$ as a subtree of $T_J$ containing $x$. 
If $G_Y$ is   virtually cyclic, then $G_\eps$ is obviously finitely generated, so we can assume that $G_Y$ is a
maximal parabolic group $P_i$. 

Since $G_Y=P_i$ is elliptic in $T_J$  and leaves $Y$ invariant, it fixes a vertex $y\in Y$. 
If $y=x$, then $G_Y\subset G_x$, so $G_\eps=G_Y=P_i$ is finitely generated.
If $y\neq x$, let $e$ be the  initial edge of the segment $[x,y]$  in $T_J$.   
It is contained in $Y$, so $G_e\subset G_Y$, hence $G_e\subset G_x\cap G_Y
\subset G_x\cap G_y \subset G_e$.
It follows that  $G_\eps=G_x\cap G_Y=G_e$ is finitely generated. 
\end{proof}

By Proposition \ref{gvqc}, vertex groups of $\Tcan$ are relatively quasiconvex, hence relatively hyperbolic.  We make the parabolic structure explicit.

\begin{lem} \label{rh} 
If $x\in V_0(\Tcan)$, the group 
$G_x$ is hyperbolic relative to the finite family of finitely generated subgroups $\calq_x=\Inc_x\cup  \calp_{ | G_x}$, where $\Inc_x$ is a set of representatives of conjugacy classes of  incident edge stabilizers and $\calp_{ | G_x}$ is the induced structure   (see  Definitions  \ref{iv} and \ref{indu}).
\end{lem}
 
\begin{proof} By   Theorem 9.1 of  \cite{Hruska_quasiconvexity}, we have to consider infinite groups of the form $G_x\cap gP_ig\m$. Recall that stabilizers of vertices adjacent to $x$ are maximal elementary subgroups, and distinct maximal elementary subgroups have finite intersection.
Thus  $gP_ig\m$ must be the stabilizer of an adjacent vertex, or have  $x$ as unique fixed point,    so $G_x\cap gP_ig\m$ is conjugate  in $G_x$ to a group in $\Inc_x\cup  \calp_{ | G_x}$.  Conversely, a group in $\Inc_x\cup  \calp_{ | G_x}$ which is not an infinite group    of the form $G_x\cap gP_ig\m$  is finite  or is   a loxodromic  maximal virtually cyclic subgroup of $G_x$. Such groups may be added to the parabolic structure.
\end{proof}

\subsection{Rigid groups have finitely many automorphisms} \label{rig}
\begin{SauveCompteurs}{thm_sci}
  \begin{thm} \label{thm_sci} 
   Let $G$ be hyperbolic relative to finitely generated subgroups $\calp=\{P_1,\dots,P_n\}$, 
 with $P_i\ne G$. 
Let $\calh=\{H_1,\dots,H_q\}$ be another family of finitely generated subgroups. If $\Out(G;\calp,\mk\calh)$ is infinite, then $G$ splits over an elementary subgroup relative to $\calp\cup\calh$.
  \end{thm}
\end{SauveCompteurs}

The hypothesis means that there are infinitely many (classes of) automorphisms which map each $P_i$ to a conjugate (in an arbitrary way) and act on each $H_j$ as conjugation by an element of $G$. 

The proof of the theorem has two steps. First, using the Bestvina-Paulin method (see \cite{Pau_arboreal}), 
extended by Belegradek-Szczepa\'nski \cite{BeSz_endomorphisms} to relatively hyperbolic groups, one constructs an action of $G$ on an $\R$-tree $T$.  Rips theory then yields a splitting. 
This is fairly standard but there are technical difficulties, in particular because the action on $T$ is not necessarily stable if the $P_i$'s are not assumed to be slender. Details are in Section \ref{sec_Rips}.

\section{Automorphisms of one-ended relatively hyperbolic groups} \label{gen}

Let   $G$ be hyperbolic relative to   $\calp=\{P_1,\dots,P_n\}$,   with $P_i$  infinite and finitely generated.  
We assume that $G$ is one-ended relative to $\calp $.  In Subsection \ref{autar} we study $\Out(G;\calp)$ through its action on the canonical JSJ tree.  This leads to the main results of this section, which are stated in Subsection \ref{aut_1bout}.
In Subsection \ref{rela} we   study automorphisms  which 
 act trivially on  another family $\calh$.

\subsection{Automorphisms of the  canonical JSJ splitting}\label{autar}

Let $\Tcan$ be  the canonical  $\Out(G;\calp)$-invariant JSJ tree as in Subsection \ref{cano} (with $\calh=\es$), 
 and $\Gcan=\Tcan/G$.    Edge stabilizers are infinite elementary subgroups. Vertex stabilizers may be 
 rigid (non-elementary), QH with finite fiber,  maximal parabolic (conjugate to a  $P_i$), or maximal loxodromic (virtually cyclic). A rigid or QH stabilizer $G_x$  fixes a unique point in $\Tcan$, hence is equal to its normalizer; incident edge stabilizers are maximal elementary subgroups of $G_x$.

We study $\Out(G; 
\calp)$ through its action on $\Tcan$ as in Subsection \ref{arb}. 
In general, $\Out(G; \calp)$ is a proper subgroup of $\Out(\Tcan)$. 

We define finite index subgroups $\Out^0 (G;\calp )$ and $\Out ^0(G;\mk\calp )$ by taking the intersection of 
$\Out (G;\calp )$ and  $\Out (G;\mk\calp )$ with the group $\Out^0(\Tcan)$ 
 consisting of automorphisms acting trivially on the graph  $\Gcan=\Tcan/G$. 

By the second  assertion of Lemma \ref{lem_virt_nobitwist}, the kernel of  $\rho:\Out^0(\Tcan)\to \prod _{v\in V} \Out(G _v)$ is the group of twists $\calt$. Note that  $\calt\inc \Out ^0(G;\mk\calp) 
  \inc \Out ^0(G;\calp) 
  $ since every $P_i $ is elliptic in $\Tcan$ and a twist acts as a conjugation on any vertex stabilizer. 

The group   $\Tw $ is the image in $\Out(G)$ of
a finite direct product  $\prod _{e\in E } 
Z_{G_{o(e)}}(G_e)$.  Each factor is virtually cyclic or contained in a conjugate of some $P_i$.  

We  now consider the image of  $\Out^0 (G;\calp )$ and $\Out ^0(G;\mk\calp )$ by $\rho_v:\Out^0(\Tcan)\to  \Out(G _v)$, for $v$ a vertex of $\Gcan$ (viewed as a vertex of $\Tcan$ with stabilizer $G_v$). 
Using Theorem \ref{thm_sci}, we shall show that both images are finite if $G_v$ is rigid. If $G_v=P_i$, the index of $\Out(G_v;\mk\Inc_v)$ in the image of 
$\Out^0 (G;\calp )$ is  finite   because $w$ is rigid whenever $e=vw$ is an edge with $\Out(G_e)$ infinite. More precisely:

\begin{prop} \label{imro} The images of  $\Out^0 (G;\calp )$ and $\Out ^0(G;\mk\calp )$ by $\rho_v:\Out^0(\Tcan)\to  \Out(G _v)$ may be described as follows:
\begin{itemize*}
\item If $G_v$ is virtually cyclic or rigid, both images are finite.
\item
If $G_v$ is a QH vertex stabilizer, 
both images contain $\Out(G_v;\mk\calb_v)$
with finite index. \item   
If $G_v$ is (conjugate to) $P_i$, the image of   $\Out^0(G;\mk\calp)$ is trivial. 
 The image of $\Out^0(G;\calp)$ 
contains  $\Out(G_v;\mk\Inc_v) $ with finite index.
\end{itemize*}
 If $e$ is any edge of $\Tcan$, 
the images of  $\Out^0 (G;\calp )$ and $\Out ^0(G;\mk\calp )$ by $\rho_e:\Out^0(\Tcan)\to  \Out(G _e)$ are finite.
\end{prop}

See Definition \ref{bv} for the definition of   the family of  full boundary subgroups $\calb_v$,  and recall that $\calb_v= \Inc_v$ if no $P_i$ is virtually cyclic. 

\begin{proof}

  If $G_v$ is virtually cyclic,    $\Out(G_v)$ is finite.
  If $G_v$ is rigid, finiteness follows from Theorem \ref{thm_sci}, as we now explain. We have seen (Lemma \ref{rh}) that $G_v$ is hyperbolic relative to a finite family  
$\calq_v=\Inc_v\cup\calp_{|G_v}$ consisting of  incident edge groups   and  conjugates of the $P_i$'s having $v$ as unique fixed point.   These groups are finitely generated by Lemma \ref{fingen}.
 By Lemma \ref{automind},  the group $\rho_v(\Out^0 (G;\calp))$ is contained in $\Out(G_v;\calq_v)$.
If it is infinite, 
$G_v$ splits over an elementary subgroup relative to $\calq_v$ by Theorem \ref{thm_sci} 
(applied with $\calp=\calq_v$  and $\calh=\es$). 
By Lemma \ref{extens}, this splitting may be used to refine $\Tcan$, yielding an elementary  splitting of   $G$ relative to $\calp$ in which $G_v$ is not elliptic. This contradicts rigidity.

 If $G_v$ is    QH,
first note that 
$\Out(G_v;\mk\calb_v)=\Out(G_v;\mk\Inc_v, \mk\calp_{|G_v})$  and 
$\Out(G_v; \calb_v)=\Out(G_v; \Inc_v, \calp_{|G_v})$ because  $\calb_v=\Inc_v\cup \calp_{|G_v}$ 
(see Subsection \ref{cano},  recalling that groups in $\calp$ are assumed to be infinite).

Lemma \ref{automind} then yields
$$\Out(G_v;\mk\calb_v)\subset 
\rho_v(\Out^0 (G;\mk\calp))
\subset \rho_v(\Out^0 (G;\calp))
\subset\Out(G_v;\calb_v).$$
We conclude by observing that $\Out(G_v;\mk\calb_v)$ has finite index in $\Out(G_v;\calb_v)$
because groups in $\calb_v$ are virtually cyclic, hence have finite outer automorphism group.

If $G_v$ is (conjugate to) $P_i$, the image of   $\Out^0(G;\mk\calp)$ is clearly trivial. 
Since $\calp_{|G_v}$  equals $\{G_v\}$ or is empty, the image of $\Out^0(G;\calp)$ 
contains  $\Out(G_v;\mk\Inc_v)=\Out(G_v; \mk\Inc_v, \calp_{|G_v})$, and we have to show that the index is finite. 
If $\Out(G_\varepsilon)$ is finite for every incident edge $\varepsilon$, 
this follows from  Lemma \ref{automind}. 
In general, we have to control the action of automorphisms on   $G_\varepsilon$ for incident edges $\varepsilon$ with $\Out(G_\varepsilon)$  infinite. Note that   there is no natural map from $\Out(G_v; \Inc_v)$ to $\Out(G_\varepsilon)$  if $N_{G_v}(G_\varepsilon)$ acts non-trivially on $G_\varepsilon$.

 Infiniteness of $\Out(G_\varepsilon)$ implies that the other endpoint of $\varepsilon$ is a rigid vertex $x$: it cannot be QH since $G_\varepsilon$ would then be virtually cyclic. As explained above,   the image of $\Out^0(G;\calp)$ in $\Out(G_x)$ is finite. Any  $\Phi\in\Out^0(G;\calp)$ has a representative $\alpha$ leaving $G_x$ and $G_v$ invariant  (the associated map $f_\alpha$ fixes $\varepsilon$). Replacing $ \Outu^0(G;\calp)$ by a finite index subgroup, we may suppose that 
$\alpha$ acts on $G_x$ as conjugation by some element $g$. This $g$  must be in $G_x$ since $G_x$ equals its normalizer, and in fact in $G_\varepsilon$ because $G_\varepsilon$ is almost malnormal in $G_x$. This shows that $
\Phi$ 
maps into $\Out(G_v;\mk G_\varepsilon)$. Arguing in this way for each incident edge proves the result  for the image of $\rho_v$. 

 Since any edge of $\Tcan$ has a vertex $v$ with $G_v$ virtually cyclic or conjugate to a $P_i$, the previous argument also shows finiteness for images by $\rho_e$. 
\end{proof}

 \begin{dfn}  Let $\Sigma$ be a compact 2-dimensional  hyperbolic orbifold. 
  The \emph{extended mapping class group} $MCG^*(\Sigma)$ is  the group of outer automorphisms of $\pi_1(\Sigma )$ preserving the set of boundary subgroups. 
  
If  $v$ is  
 a QH vertex of $\Tcan$ with underlying orbifold $\Sigma$ and finite fiber $F$,
we define  $MCG^0_{\Tcan}(\Sigma )$   as the group
  $\Out(G_v;\mk\calb_v)$. 
This group   
maps with finite kernel  onto a finite index subgroup of   $MCG^*(\Sigma)$   \cite[pp.\ 240 and 268]{DG2}. 
\end{dfn}

As noted in \cite{Fujiwara_outer,DG2}, 
one can understand the extended mapping class group of an orbifold with mirrors
in terms of the extended mapping class group of a suborbifold without mirrors.

If $G$ is torsion free, $\Sigma $ is a surface,  $MCG^*(\Sigma)$ is the group of isotopy classes of homeomorphisms, and $MCG^0_{\Tcan}(\Sigma )$ is  the group of isotopy classes of homeomorphisms   which map each boundary component to itself in an orientation-preserving way   (in this case it only depends on $\Sigma$ since the fiber is trivial).

Mapping class groups are usually infinite, but there are exceptions. In the torsion-free case, the exceptions are 
the pair of pants and the twice punctured projective plane  \cite[Cor 4.6]{Korkmaz_MCG};  
all other hyperbolic surfaces contain
an essential 2-sided simple closed curve not bounding a M\"obius band, so there is  a Dehn twist of  infinite order. 
As a QH vertex, a pair of pants is rigid; every simple closed curve is homotopically trivial or boundary parallel. A twice punctured projective plane is flexible, but every 2-sided simple closed curve is homotopically trivial, boundary parallel, or bounds a M\"obius band, so there is no non-trivial Dehn twist.

\subsection{Automorphisms of  $G$}\label{aut_1bout}

Motivated by the previous subsection, we define a subgroup $\Out^1(G;\mk\calp)\subset \Out^0(G;\mk\calp)$ as the set of $\Phi$ such that  $\rho_v(\Phi)$ is trivial  if  $G_v$ is virtually cyclic, rigid, or conjugate to a $P_i$,
and    $\rho_v(\Phi)\in\Out(G_v;\mk\calb_v)=MCG^0_{\Tcan}(\Sigma )$ if $G_v$ is QH.  Proposition \ref{imro} shows that this subgroup    has finite index. 
We define a finite index subgroup
$\Out^1(G; \calp)\subset \Out^0(G; \calp)$ similarly,  allowing $\rho_v(\Phi)\in \Out(G_v;\mk\Inc_v)$
if $G_v$ is conjugate to a $P_i$.

 We may now sum up the  discussion in Subsection \ref{autar} as:

  \begin{thm}\label{thm_struct_m}
  Let $G$ be hyperbolic relative to 
  $\calp=\{P_1,\dots,P_n\}$,  with $P_i$ infinite and  finitely
    generated. Assume that $G$ is one-ended relative to $\calp$.

Then  $\Out(G;\mk\calp)$ and $\Out(G;\calp)$ have   finite index subgroups
  $\Out^1(G;\mk\calp)$  and $\Out^1(G;\calp)$ which fit  in   exact sequences
$$1\ra \Tw \ra   \Out^1(G;\mk\calp) \ra \prod_{i=1}^p MCG^0_{\Tcan}(\Sigma_i)\ra 1$$
$$1\ra \Tw \ra   \Outu^1(G;\calp) \ra \prod_{i=1}^p MCG^0_{\Tcan}(\Sigma_i) \times \prod _j \Out(P_j;\mk\Inc_{P_j})\ra 1$$
where:
\begin{itemize*}
\item $\Tw $ is the group of twists of the canonical elementary JSJ decomposition $\Tcan$ relative to $\calp$; it is   a quotient  of a finite direct product where each factor is  a subgroup of $G$ which is virtually cyclic or contained in a $P_i$;

\item
$\Sigma_1,\dots,\Sigma_p$ are the 2-orbifolds occuring in  flexible QH vertices   $v$ of $\Tcan$, and $MCG^0_{\Tcan}(\Sigma_i)$ 
maps with finite kernel onto a finite index subgroup of the extended mapping class group   $MCG^*(\Sigma_i)$; 
\item 
the last product is taken only over those $P_j$'s which occur as vertex stabilizers of $\Tcan$, and $\Inc_{P_j}$ is the set of incident edge groups.  \qed
\end{itemize*}

\end{thm}

 Note that $\Tw$ is   slender   (resp.\ small, virtually solvable, virtually nilpotent, virtually abelian) if the $P_i$'s are. Also note that $\Outu(G;\calp)$ has finite index in $\Out(G)$ if the $P_i$'s are small but not virtually cyclic, since they may be characterized up to conjugacy as   maximal among the subgroups of $G$ which are small but not virtually cyclic. More generally, this holds if  no $P_i$ is relatively hyperbolic  \cite[Lemma 3.2]{MiOs_fixed}.

Recall that $G$ is \emph{toral relatively hyperbolic} if it is torsion-free  and hyperbolic relative to a finite family   $\calp$ of finitely generated abelian  subgroups. 
Limit groups, and more generally groups acting freely on $\R^n$-trees, are toral relatively hyperbolic \cite{Dah_combination,Gui_limit}. 

\begin{cor} \label{limgp}
  Let $G$ be toral relatively hyperbolic and  one-ended.
Then some finite index subgroup $\Out^1(G)$ of $\Out(G)$ fits in an exact sequence 
$$1\ra \Tw \ra   \Out^1(G) \ra \prod_{i=1}^p MCG^0(\Sigma_i) \times  \prod_{k=1}^m 
  GL_{r_k,n_k}(\bbZ) 
 \ra 1$$
where $\Tw$ is finitely generated free abelian,    $MCG^0(\Sigma_i)$ is  the group of isotopy classes of homeomorphisms of a compact surface $\Sigma_i$     mapping each boundary component to itself in an orientation-preserving way,  and  $GL_{r,n}(\bbZ) =M_{r,n}(\bbZ)\rtimes GL_{r}(\bbZ)$
is the group of automorphisms of $\bbZ^{n+r}$ fixing the first $n$ generators. 
\end{cor}

 See Theorem 5.3 of \cite{BuKhMi_isomorphism} and Theorem 6.5 of \cite {GL1} for the case of limit groups, based on results from \cite{KhMy_effective} and \cite{BuKhMi_isomorphism}.

\begin{proof}  We may assume that no $P_i\in\calp$ is cyclic, so $\Outu(G;\calp)$ has finite index in $\Out(G)$.
If $P_j$ is isomorphic to  $\Z^{a}$, then $\Out(P_j;\mk\Inc_{P_j})$ is isomorphic to  some  $GL_{r,n}(\bbZ)
 $ with $r+n=a$, so the exact sequence follows from Theorem \ref{thm_struct_m}. We know that the group of twists $\Tw$  of $\Tcan$  is finitely generated and   abelian. There remains to check that it is torsion-free.

 Recall  from Subsection \ref{arb} that $\Tw$  is generated by the product of all $Z_{G_{o(e)}}(G_e)$, subject to edge and vertex relations. Denoting an edge of  $\Gcan=\Tcan/G$ by
$ \varepsilon=(x,Y)$ with   $x\in V_0(\Gcan)$ and $Y\in V_1(\Gcan)$,
there is no relation at the vertex $x$ since $Z(G_x)$ is trivial. 
Moreover, $Z_{G_x}(G_\varepsilon)=G_\varepsilon$, so the edge relation identifies the twists around $\varepsilon$ near $x$ with   twists near $Y$. Thus $\Tw$ is the direct product, over vertices $v$ of 
 $ \Gcan$ carrying an  abelian group, of
$(\prod_{e\in E_v} Z_{G_v}(G_e))/Z(G_v)$ where $E_v$ is the set of  oriented edges with origin  $v$ and 
$Z(G_v)$ is embedded diagonally.
Since $G_v$ is abelian, $Z_{G_v}(G_e)= Z(G_v)=G_v$,  so $\Tw$ is isomorphic to a finite direct product  $\prod (G_v)^{ | E_v | -1}$
of   abelian vertex groups. It is therefore torsion-free. 
\end{proof}

\begin{cor} \label{VFL}
  If $G$ is a  toral relatively hyperbolic group,
  $\Out(G)$ is virtually torsion-free and  has a finite index subgroup with a finite classifying space.
\end{cor}

\begin{proof} This follows from Corollary \ref{limgp} if $G$ is one-ended. In general,   we write 
$G=G_1*\dots*G_q*F$ with $G_\ell$ one-ended and $F$ free. All groups $G_\ell$ and $G_\ell/Z(G_\ell)$ have a finite classifying space  \cite{Dah_classifying}, so 
we
can apply  Theorem 5.2 of 
\cite{GL1}.   (We mention here that the arguments given in \cite{GL1} are insufficient to get a finite classifying space: there should exist finite classifying spaces for the groups   $\Out^S(G)$ themselves (rather than for finite index subgroups); this is achieved by   restricting to some finite index subgroup of $\Out(G)$, see \cite {GL_McCool}  for details.) 
\end{proof}

\subsection{The relative case}\label{rela}

We   generalize the analysis   of Subsection \ref{autar}
to a relative situation. 

Let $G,\calp$ be as above, 
and  fix another 
family of finitely generated subgroups $\calh=\{H_1,\dots, H_q\}$.
Assume that $G$ is one-ended relative to $\calp\cup\calh$,
and 
 let now $\Tcan$ be
 the canonical elementary JSJ tree relative to $\calp\cup\calh$.

  \begin{thm}\label{thm_struct_m_rel}
Under these hypotheses,
  $\Out(G;\mk\calp,\mk\calh)$ and $\Out(G;\calp,\mk\calh)$ have   finite index subgroups
  $\Out^1(G;\mk\calp,\mk\calh)$ and $\Out^1(G;\calp,\mk\calh)$     which fit  in   exact sequences
$$1\ra \Tw \ra   \Out^1(G;\mk\calp,\mk\calh) \ra \prod_{i=1}^p MCG^1_{\Tcan}(\Sigma_i)\ra 1$$
$$1\ra \Tw \ra  \Out^1(G;\calp,\mk\calh) \ra \prod_{i=1}^p MCG^1_{\Tcan}(\Sigma_i) \times \prod_j  \Out(P_j;\mk\Inc_{P_j},\mk \calh_{|P_j})\ra 1$$
 as in Theorem \ref{thm_struct_m}.

The group  $MCG^1_{\Tcan}(\Sigma_i)$ equals $\Out(G_v;\mk\calb_v, \mk\calf_v)$,
where $\calf_v$ is a set 
of representatives of   conjugacy classes of finite subgroups in $G_v$;
it  is a finite index subgroup    of  $MCG^0_{\Tcan}(\Sigma_i)=\Out(G_v;\mk\calb_v)$.
\end{thm}

 Note that $\calf_v$ is a finite set since the QH vertex group $G_v$ maps onto a 2-orbifold group with finite kernel.   The family $\calf_v$ is not needed if all groups in $\calh$ are infinite (see the proof below).

 The theorem  is proved as in the absolute case, replacing Proposition \ref{imro} by the following result.

\begin{prop} \label{imro_rel} The images of  $\Out^0 (G;\calp,\mk\calh )$ and $\Out ^0(G;\mk\calp,\mk\calh )$ 
by $\rho_v:\Out^0(\Tcan)\to  \Out(G _v)$ may be described as follows:
\begin{itemize*}
\item If $G_v$ is virtually cyclic or rigid, both images are finite.
\item If $G_v$ is a QH vertex stabilizer, 
both images contain $\Out(G_v;\mk\calb_v,\mk \calf_v)$
with finite index (where $\calf_v$ is  as in Theorem \ref{thm_struct_m_rel}
and $\calb_v$ is as in 
Definition \ref{bv}). 

\item   
If $G_v$ is (conjugate to) $P_i$, the image of   $\Out^0(G;\mk\calp,\mk\calh)$ is trivial. 
 The image of $\Out^0(G;\calp,\mk\calh)$ 
contains  $\Out(G_v;\mk\Inc_v,\mk\calh_{|G_v}) $ with finite index.
\end{itemize*}
\end{prop}

\begin{proof}
We only mention the differences with the proof of Proposition \ref{imro}.

 If $v$ is a (non-elementary) rigid vertex,
the images of  $\Out^0 (G;\calp,\mk\calh )$ and $\Out ^0(G;\mk\calp,\mk\calh )$ 
by $\rho_v$ are contained in 
in $\Out(G_v;\calq_v,\mk\calh_{|G_v})$ by Lemma \ref{automind}.
It follows that the images are finite since, otherwise,
Theorem \ref{thm_sci} would yield a splitting relative to 
$\calq_v\cup\calh_{|G_v}$, which extends to a splitting of $G$ relative to $\calp\cup\calh$
by Lemma \ref{extens}.

 When  $G_v$ is conjugate to a parabolic group $P_j$,  Lemma \ref{automind}
says that the image of 
$\Out^0(G;\calp,\mk\calh)$ contains $\Out(G_v;\mk\Inc_ v, \mk \calh_{|G_v})$,
and the index is finite for the same reason as before. 

 When $G_v$ is QH,  
we write
$$\Out(G_v;\mk\Inc_v,\calp_{|G_v},\mk\calh_{|G_v})
\subset \rho_v (\Out^0(T;\calp,\mk \calh))
\subset  
\Out(G_v;\Inc_v,\calp_{|G_v},\mk\calh_{|G_v})$$
using Lemma \ref{automind}.

The proof in the non-relative case relied on   the equality  $\calb_v=\Inc_v\cup \calp_{|G_v}$. 
Here    (see Subsection \ref{cano}) we have $\Inc_v\cup \calp_{|G_v}\inc \calb_v$, and a group  $B\in\calb_v$ not in $\Inc_v\cup \calp_{|G_v} $ contains with finite index a group $H'$ conjugate to some $H\in \calh_{|G_v} $. Since $B$ is the only maximal elementary subgroup of $G_v$ containing $H'$, any automorphism preserving $H'$ preserves $B$, so   $\Out(G_v;\Inc_v,\calp_{|G_v},\mk\calh_{|G_v})\subset \Out(G_v;\calb_v)$.

If all groups in $\calh$ are infinite, the intersection of any conjugate of $H_j$ with $G_v$ is contained in a full boundary subgroup, so $\Out(G_v;\mk\calb_v )\subset \Out(G_v;\mk\Inc_v,\calp_{|G_v},\mk\calh_{|G_v})$.
Otherwise 
$\calh_{|G_v}$ may contain   finite groups (fixing  $v$ but  no other vertex of  $\Tcan$) and we can only write $\Out(G_v;\mk\calb_v,\mk \calf_v )\subset \Out(G_v;\mk\Inc_v,\calp_{|G_v},\mk\calh_{|G_v})$.
The proposition follows because the index of $\Out(G_v;\mk\calb_v,\mk \calf_v) $ in $ \Out(G_v;\calb_v)$ is finite.
\end{proof}

 \begin{rem} \label{rips}
 Because we use  Theorem \ref{thm_sci} to control automorphisms of rigid groups,
we do not have a similar result concerning $\Out(G;\mk \calp\cup\calh)$ or 
$\Out(G;\calp\cup\calh)$: 
we have to impose that automorphisms act trivially on $\calh$.  We also need finite generation of groups in $\calh$.
\end{rem}

Arguing as in the previous subsection, one gets:

\begin{cor}\label{cor_MCtoral}
Let $G$ be toral relatively hyperbolic, one-ended relative to a family $\calh=\{H_1,\dots, H_q\}$ of finitely generated subgroups. Then 
$ \Out(G; \mk\calh) $ has a finite index subgroup $ \Out^1(G;\mk\calh)$  fitting in an exact sequence
as in Corollary \ref{limgp}. \qed
\end{cor}

\section{The modular group}\label{modul}

 The goal of this  section is to show that the modular group, usually defined by considering all suitable splittings of  a group $G$, may be seen on   a single splitting, namely the canonical JSJ decomposition.

\subsection{Definitions and examples}

    Let $G$ be hyperbolic relative to $\calp=\{P_1,\dots, P_n\}$, where each $P_i$ is 
    finitely generated. Without loss of generality, we   assume that no $P_i$ is 
virtually cyclic (in particular, $P_i$ is infinite). Let $\calh$ be another finite   family of finitely generated subgroups  $H_j$ such that every $P_i$ which contains  a free group $F_2$ is contained in a group of $\calh$. 

In particular, we may take $\calh=\calp$, or $\calh=\es$ if every $P_i$ is small. 
 We will assume that $G$ is one-ended relative to  $ \calh$ (equivalently, relative to $\calp\cup\calh$  since  every $P_i$ is one-ended or contained in a group of $\calh$).    

    We consider  trees $T$ 
  with elementary edge stabilizers,
  which are relative to $\calh$  (universal ellipticity will be with respect to these trees,  unless indicated otherwise). They are not necessarily relative to $P_i$ if $P_i$ is small, but our assumption on $\calh$ implies that 
 elementary subgroups   which are not small 
 have a   conjugate   contained in a group in $\calh$,
 so  are  universally elliptic.    
  We shall associate a modular group $\mo(T)\inc\Out(T)\inc\Out(G )$     to such a  tree  $T$.

 \begin{lem}\label{lem_elemQH}
 If $v$ is a flexible  QH vertex with finite fiber, then any elementary subgroup of $G_v$ is virtually cyclic. 
\end{lem}

Recall (Definition \ref{dqh}) that  $G_v$  maps onto 
$\pi_1(\Sigma)$ with finite kernel $F$; flexibility of $G_v$ is equivalent to the 2-orbifold $\Sigma$ containing an essential 1-suborbifold.  Since $T$ is only assumed to be relative to $\calh$, Definition \ref{dqh} only restricts intersections of $G_v$ with conjugates of groups in $\calh$.
 
\begin{proof}
 Assume on the contrary that some subgroup $E<G_v$ is elementary, but not virtually cyclic.
Then $E$ contains $F_2$ and is parabolic.  As pointed out before, our assumption on $\calh$ implies that 
$E$ is universally elliptic. This contradicts Remark \ref{rem_UEQH}
saying that such a group has to be virtually cyclic.
\end{proof}

    \begin{rem*}  If $G_v$ is a flexible QH vertex stabilizer with elementary fiber $F$, then $F$ is finite. Indeed,  if $F $ is infinite, then $G_v$ is elementary by almost malnormality of maximal elementary subgroups (Corollary \ref{almar}). Since it contains $F_2$, it is universally elliptic, contradicting flexibility.    
 \end{rem*}   

  \begin{dfn}    We say that a vertex $v$ of $T$ (or of $\Gamma=T/G$) is \emph{modular} if  $G_v$ is flexible and QH (relative to $\calh$) with finite fiber,  
  or $G_v$ is elementary. Note that  $G_v$ cannot be both.
   \end{dfn}

    Recall (Subsection \ref{arb}) the maps $\rho_v: \Out^0(T)\to \Out(G_v)$ defined on the finite index subgroup of $\Out(T)$ consisting of automorphisms acting trivially on  $\Gamma=T/G$. 
    
    \begin{dfn}
     We define $\mo(T)$   by saying that $\Phi\in\Out^0(T)$ belongs to $\mo(T)$ 
if  it satisfies the following conditions:
    
    \begin{itemize}
    \item If $v$ is not modular, $\rho_v(\Phi)$ is trivial.
    \item If $G_v$ is elementary, $\rho_v(\Phi)\in\Out(G_v;\mk\Inc_v)$; in other words, $\rho_v(\Phi)$ acts on each incident edge group  as a conjugation. 
    \item  If $G_v$ is   QH   with finite fiber, and flexible, then $\rho_v(\Phi)\in\Out(G_v;\mk\calb_v)$, with $\calb_v$ consisting of  full preimages of boundary subgroups of $\pi_1(\Sigma)$ as in  Definition \ref{bv}. 
    \end{itemize}
    \end{dfn}
    
 Note that $\mo(T)$ contains the group of twists $\Tw(T)$,   and that   automorphisms in $\Mod(T)$ need
not preserve $\calh$.
 
    We  have assumed that $G$ is one-ended relative to $\calh$, 
so we can consider 
the canonical  elementary JSJ
tree $\Tcan$ relative to $\calp\cup\calh$ as in  Subsection \ref{cano}.
Note that $\Mod(\Tcan)$ has finite index in $\Out(G; \calp)$ when $\calh=\calp$ 
 (it contains the group $\Out^1(G;\calp)$ defined in 
 Subsection  \ref{aut_1bout}, possibly strictly because of  vertices with $G_v$ virtually cyclic).

    \begin{thm}  \label{thm_mod}
      Let $G$ be hyperbolic relative to $\calp=\{P_1,\dots, P_n\}$, with each $P_i$  
    finitely generated, not virtually cyclic. Let $\calh$ be a  family of subgroups such that every $P_i$ which contains $F_2$ is contained in a group of $\calh$, and $G$ is one-ended relative to $\calh$. 
    
    If $T$ is any elementary splitting of $G$ relative to $\calh$, then $\mo(T)\inc\mo(\Tcan)$, where $\Tcan$ is  the canonical elementary JSJ tree relative to $\calp\cup\calh$.
 \end{thm}
    
    This applies in particular if $G$ is one-ended and no $P_i$ contains $F_2$ (taking $\calh=\es$), or if $G$ is one-ended relative to an arbitrary $\calp$ and we restrict to splittings relative to $\calp$ (taking $\calh=\calp$).

        \begin{rem} Rather than defining $\mo(T)$ by imposing conditions on the action on vertex groups, as we just did, one could define it by giving generators: twists around edges, and certain automorphisms of vertex groups. This would yield a slightly smaller group $\mo'(T)$: its intersection with $\ker\rho$ is $\Tw $, whereas $\mo(T)$ contains all of $\ker \rho$. Theorem \ref{thm_mod}  (and Theorem \ref{thm_modCSA} below) also hold  with this more restrictive definition,  since $\mo'(\Tcan)=\mo(\Tcan)$      by     Assertion 2 of  Lemma \ref{lem_virt_nobitwist}. 
        \end{rem}
    
    There is a similar statement for torsion-free CSA groups (recall that $G$ is CSA if centralizers of non-trivial elements are abelian and malnormal).  We now consider abelian splittings of $G$. A vertex  $v$ is modular if $G_v$ is either abelian or QH as above (in this case $F$ is trivial and $\Sigma$ is a surface). The definition of $\mo(T)$ is the same (with elementary replaced by abelian). The tree $\Tcan $ is the 
 canonical     abelian  JSJ tree relative to non-cyclic abelian subgroups; it is also the  tree of cylinders of the (non-relative) abelian JSJ deformation space (see Theorem 11.1 of \cite{GL3b}). 
    
       \begin{thm}  \label{thm_modCSA}  Let $G$ be a finitely generated, torsion-free, one-ended, CSA group.     If $T$ is any   splitting of $G$ over abelian groups, then $\mo(T)\inc\mo(\Tcan)$, where $\Tcan$ is  the tree of cylinders of the abelian JSJ deformation space.
    \end{thm}

\begin{example}
Let $G$ be the Baumslag-Solitar group $BS(2,4)=\grp{a,b\mid ba^2b\m=a^4}$.
Any splitting of $G$ as a graph of infinite cyclic groups is a cyclic JSJ decomposition of $G$ 
\cite{For_uniqueness,GL3a}.
Its modular group coincides with its group of   twists, and is a finite abelian group (see \cite{Lev_GBS}).
But the JSJ deformation space of $BS(2,4)$ is quite large \cite{Clay_deformation},
and JSJ splittings of $BS(2,4)$ may have   modular groups of arbitrarily large order. In particular, 
  there is no   splitting whose modular group contains all others.
\end{example}

\begin{example}
Even if $G$ is as in Theorem \ref{thm_mod},
one cannot replace $\Tcan$ by an arbitrary tree   in its deformation space:
there exists such trees whose modular group is not maximal (even up to finite index).

Indeed, let  $G=A_1*_{C_1}B*_{C_2} A_2$, where:  
\begin{itemize*}
\item  $A_1$ and $A_2$ are torsion-free hyperbolic groups with no cyclic splitting;
\item  $C_i$ is a maximal infinite cyclic subgroup of $A_i$;
\item  $B$ is torsion-free, hyperbolic relative to a subgroup $\Hat C= C_1\oplus C_2\oplus \bbZ\simeq\Z^3$, and  does not split over an   abelian group. 
\end{itemize*}
The group $G$ is hyperbolic relative to $\Hat C$ by \cite{Dah_combination}.

The graph of groups  
$$\xymatrix{ & B \ar@{-}[d]^{C_1\oplus C_2}\\
A_1\ar@{-}[r]_{C_1\hspace{0.4cm}  }&C_1\oplus C_2 \ar@{-}[r]_{\hspace{0.4cm} C_2}&A_2}$$
is an elementary JSJ decomposition of $G$ (both absolute and relative to $\Hat C$) because its vertex groups are universally elliptic (see Lemma 4.7 of \cite{GL3a}). Given any $z\in \Hat C\setminus (C_1\oplus C_2)$, the automorphism  $\tau$ defined as the identity on $A_1$ and $B$ and as   conjugation by $z$
 on $A_2$   is not an automorphism of this graph of groups. But $\tau$ is a twist of $\Tcan$, which is the Bass-Serre tree of the graph of groups below.
$$\xymatrix{ & B \ar@{-}[d]^{\Hat C}\\
A_1\ar@{-}[r]_{C_1
  }&\Hat C \ar@{-}[r]_{
C_2}&A_2}$$
\end{example}

 \subsection{Proof of Theorem \ref{thm_mod}}

We prove  Theorem \ref{thm_mod}. 
The proof of Theorem \ref{thm_modCSA} is similar and left to the reader.
  The main difference in the context of a general CSA group is that we have no version of Theorem
\ref{thm_sci} saying that a rigid vertex group only has finitely many outer automorphisms.
But the proof  given below  does not use 
Theorem \ref{thm_sci}.

By Theorem 13.1 of \cite{GL3b}, the trees $\Tcan $ and $T$ are compatible: they have a common refinement $\hat T$  (as defined in Subsection \ref{defsp}). We may assume that no edge of $\hat T$ is collapsed in both $\Tcan $ and $T$ (so $\hat T$ is the lcm of $\Tcan $ and $T$ as defined in Section 3 of \cite{GL3b}).  The tree $\hat T$ has elementary edge stabilizers and is relative to $\calh$  since $\Tcan $ and $T$ are (Proposition 3.22 of \cite{GL3b}).

We first claim that $\Tcan $ is $\mo(T)$-invariant. To see this, it suffices to show that the image of an infinite group  $J\in \calp\cup\calh$ by a modular automorphism $\Phi$ has a finite index subgroup which is contained (up to conjugacy) in a group 
belonging to $\calp\cup\calh$.

 If $J$ is a  small  $P_i$, its image is elementary and  not virtually cyclic, so is parabolic. If $J $ is a $P_i$ containing $F_2$, it is contained in a group of $\calh$, so we only have to consider groups   $J\in\calh$. Such a group  fixes a vertex $v$ in $T$, and  $G_v$ is $\Phi$-invariant.
 We distinguish three cases. 

If $v$ is non-modular, $\Phi$ acts trivially on $J$.  
 If $v$ is a flexible QH vertex, then $J$ is contained in a group of $\calb_v$, hence $\Phi$-invariant. Now suppose that $J$ is contained in an elementary $G_v$. If $G_v$ contains $F_2$, it is a $\Phi$-invariant group contained in a group of $\calh$, so the image of $J$ is contained in a group of $\calh$.   The case when $G_v$ is small but not virtually cyclic has been dealt with before. If $G_v$ is virtually cyclic, $\Phi(J)$ has a finite index subgroup contained in $J$.
This completes the proof of  the claim.

Let $\Phi\in\mo(T)$. 
   The heart of the proof  of the theorem is to  study the action of $\Phi$ 
on a non-elementary vertex stabilizer $G_v$ of $\Tcan$ (it is rigid or QH). 
In particular, given an edge $e=vw$ in $\Tcan$, we show that $\Phi$ has a representative $\alpha$ leaving $G_v$ invariant and equal to the identity on $G_e$.

We have defined $\Tcan$ as a JSJ tree relative to $\calp\cup\calh$. When $\Tcan$ is viewed as relative to $ \calh$ only, a flexible QH vertex remains flexible QH. 
It follows from Sections 8  and 13 of \cite{GL3b} that a rigid (non-elementary) vertex stabilizer $G_v$ of $\Tcan$  remains universally elliptic  relative to $\calh$. The argument goes as follows:
if $T$ is an elementary splitting relative to $\calh$,   
then $G_v$ is elliptic in its tree of cylinders $T_c$ because $T_c$ is relative to $\calh\cup\calp$ and $G_v$ is rigid  relative to $\calh\cup\calp$;  since 
 groups elliptic in $T_c$ but not in $T$ are elementary (see Subsection \ref{defcyl}), and 
$G_v$ is not, 
this implies that $G_v$ is elliptic in $T$.

We distinguish two cases.

Case 1: $v$ is rigid.   
Then $G_v$ is universally elliptic,  so fixes a   point $u$ in $T$, 
which is unique because  edge stabilizers are elementary  and $G_v$ is not.  
The group $G_u$ cannot be elementary.  By Remark \ref{rem_UEQH},
it cannot be flexible QH because its subgroup $G_v$ is universally elliptic and non-elementary. 
Thus   $u$ is not modular and $\Phi$ has a representative $\alpha$ equal to the identity on $G_u$, hence on $G_v$. 

Case 2: $v$ is a flexible QH vertex  
of $\Tcan$. 
Let $e $ be an adjacent edge and $\hat e$ its lift to $\hat T$. Recall that $G_e=G_{\hat e}$ is a maximal elementary subgroup of $G_v$. We define a point $\hat v\in \hat T$ as follows. If $G_v$ is   elliptic in $\hat T$, we call $\hat v$ its unique fixed point. If it is not elliptic, its  action   on its minimal subtree in $\hat T$ is dual to a family of 1-suborbifolds of $\Sigma$   (see Lemma 7.4 of \cite{GL3a}).  We let $\hat v$ be the point of that subtree closest to $\hat e$ (possibly an endpoint of $\hat e$).

The stabilizer of $\hat v$ is QH, associated to a   suborbifold $\hat \Sigma$ of $\Sigma$ (if $\hat \Sigma$ contains no essential $1$-suborbifold, $\hat v$ is a rigid vertex of $\hat T$). 
Note that $G_{\Hat v}$ is non-elementary   by Lemma \ref{lem_elemQH}. 
The stabilizer of $\hat e$, and also of edges between $\hat e$ and $\hat v$ if any,  
 is contained in $G_{\hat v}$, in fact in  the preimage of a boundary subgroup of $\pi_1(\Sigma)$ 
and $\pi_1(\hat\Sigma)$. 

Let $u$ be the image of $\hat v$ in $T$. 

Subcase 2a: $u$ is not modular. Then $\Phi$ has a representative $\alpha$ equal to the identity on $G_{u}$, 
hence on $G_{\hat v}$ and  on $G_e=G_{\hat e}$. 
Note that $\alpha$ leaves $G_v$ invariant because $\alpha$ is an automorphism of  $\Tcan$ 
and $v$ is the only vertex of $\Tcan$ fixed by $G_{\hat v}$. 

Subcase 2b: $G_{u}$ is QH with finite fiber. 
Then $G_{u}$ is elliptic in $\Tcan$ (see Proposition 7.13 of \cite{GL3a}, which is valid in a relative setting),
hence in $\hat T$ \cite[Proposition 3.22]{GL3b}, so $G_{u}=G_{\hat v}$. 
 Unfortunately, this argument says nothing about incident edge groups at $u$.

First suppose that $\hat v$ is an endpoint of $\hat e$.
 Choose an edge path   with origin $\hat v$, starting with  $\hat e$, such that all edges except the last one get collapsed to $ u$ in $T$ (this path consists of the single edge $\hat e$ if $\hat e$ 
 is not collapsed to a point in $T$). Call this last edge $\hat e'$, and its initial vertex $a$. We have $G_{\hat e'}\inc G_a\inc G_{u}=G_{\hat v}$, so $G_{\hat e'}\inc G_{\hat e}$.

The group $G_{u}$ is QH with finite fiber, and 
 $G_{\hat e'}$ is   an incident edge group. It is  infinite by one-endedness, so is contained in a unique maximal elementary subgroup $C$ of $G_{u}$ (the preimage of a boundary subgroup of the underlying 2-orbifold). Since 
 $G_{\hat e'}\inc G_{\hat e}$, we have $G_{\hat e}\inc C$. By definition of $\mo(T)$, there is a representative $\alpha$ of $\Phi$ leaving $G_{u}$  invariant and equal to the identity on $C$, hence on $G_e=G_{\hat e}$. As above, $\alpha$ leaves $G_v$ invariant.

If there are edges between $\hat v$ and $\hat e$, call $\hat e'$ the edge that contains $\hat v$. It is not collapsed to a point  in $T$, since it is collapsed in $\Tcan$, so  $G_{\hat e'}$ is   an incident edge group of $G_{u}$. We now have $G_{\hat e }\inc G_{\hat e'}\inc C$ and we argue as in the previous case. This completes the analysis of the action of $\Phi$ on $G_e$   in case 2.
 
Still in case 2 (\ie assuming that $v$ is a flexible QH vertex stabilizer of $\Tcan$), 
we also  need to understand the action of $\Phi$ on an element $B\in\calb_v$ which is not an incident edge stabilizer. 
Such a $B$ contains a conjugate of an $H_j$ with finite index. 

By minimality of $\hat T$, the group $B$ fixes a QH vertex $\hat v\in\hat T$. We then argue as above. 
In subcase 2b, we have $B\in\calb_u$ (up to conjugacy) because $B$ contains a conjugate of $H_j$, so we can find $\alpha$  leaving $G_{u}$  invariant and equal to the identity on $B $ since $\Phi\in \mo(T)$.
  This finishes case 2.

We can now conclude. Consider $\Gcan=\Tcan/G$, and recall that $\Tcan$ is a tree of cylinders, so $\Gcan$ is   bipartite, with edges joining a vertex $x\in V_0(\Gcan)$ carrying a non-elementary group to a vertex $Y\in V_1(\Gcan)$  carrying an elementary group. We know that $\Phi$ fixes each vertex $x\in V_0(\Gcan)$, 
and its action on $G_x$ is trivial if $x$ is not modular, 
in $\Out(G_x;\mk\calb_x)$ if $G_x$ is QH. 

Since $\Tcan$ is a tree of cylinders, distinct edges of $\Gcan$ with origin $x$  in $V_0(\Gcan)$ carry groups which are not conjugate in $G_x$.    As $\rho_x(\Phi)\in\Out(G_x; \mk\Inc_x)$, we deduce that $\Phi$ acts as the identity on edges of $\Gcan$ with origin $x$, hence on the whole of   $\Gcan$. Thus $\Phi\in\Out^0(\Tcan)$. 

There remains to check that 
$\rho_Y(\Phi)\in\Out(G_Y;\mk\Inc_Y)$ for $Y\in V_1(\Gcan)$. If $\varepsilon=(x,Y)$ is an adjacent edge, we have seen that $\Phi$ has a representative $\alpha$ equal to the identity on $G_\varepsilon$. Since $G_\varepsilon$ is infinite, $G_Y$ is the unique maximal elementary subgroup containing it, so $\alpha$ leaves $G_Y$ invariant. This completes the proof. 

\section{Induced automorphisms}  \label{sec_induced}

 In  this section $G$ is hyperbolic relative to $\calp=\{P_1,\dots, P_n\}$, and we study automorphisms of $P_1$ which are induced by automorphisms of $G$. We  then apply this to the case when $G$ is hyperbolic and $H$ is a malnormal quasiconvex subgroup, viewing $G$ as hyperbolic relative to $H$.

\begin{dfn}  Given families of subgroups $\calp$ and $\calh$, and a subgroup $Q$, 
  we say that $\alpha\in\Aut(Q)$ is \emph{extendable to $G$ relative to
    $\calp$ and $\mk\calh$} if it is the restriction  to $Q$ of an automorphism of  $G$ representing an element of
  $\Out(G;\calp,\mk\calh)$.  

 Being extendable only depends   on the image of $\alpha$ in $\Out(Q)$, 
so we   define the \emph{group of extendable automorphisms} $\Out(Q \fleche (G;\calp, \mk\calh))\inc\Out(Q )$. We write $\Out(Q \fleche (G;\calp ))$ when $\calh=\es$, and $\Out(Q \fleche G)$ for   $\Out(Q \fleche (G;\es))=\Out(Q \fleche (G;\{Q \} )) $.

If  $Q$ equals its normalizer (for instance if $Q$ is an infinite maximal parabolic subgroup), there is a map    $ \Out(G;Q)\to\Out(Q)$, and $\Out(Q \fleche G)$ is its image. 
\end{dfn}

Suppose   that 
 $P_1=G_v$ is a vertex group of a splitting of $G$ relative to $\calp=\{P_1,\dots, P_n\}$, 
and $P_1$ contains no conjugate of $P_i$ for $i>1$  (this is automatic if no $P_i$ is   finite). 
Then $\Out(P_1\fleche (G;\calp))$  contains  $\Out (G_v;\mk\Inc_v)$  (see Lemma \ref{automind}). 
The following theorem says that virtually all extendable automorphisms occur in this fashion.

 \begin{thm}\label{caleg_general} 
  Let $G$ be hyperbolic relative to   $\calp=\{P_1,\dots, P_n\}$, with $P_i$ finitely generated and infinite, 
 and  $P_i\ne G$. Let  $\calh$ be a finite family of finitely generated subgroups of $G$. 
If $\Out(P_1\fleche (G;\calp,\mk\calh))$ is infinite,   then:
  \begin{enumerate}
  \item

 $P_1$ is a vertex group $G_v$  in 
an elementary JSJ decomposition  $\Gamma$  relative to $\calp\cup\calh$. Edge groups of $\Gamma$ are finitely generated.

  \item
The group $\Out(P_1\fleche (G;\calp,\mk\calh))\inc\Out(P_1)$ has a finite index subgroup equal to  
$\Out(P_1;\mk\calk)$, where  $\calk=\Inc_v\cup\calh_{|G_v}$  is a finite family   of finitely generated 
subgroups of $P_1$ 
(the family of incident edge groups $\Inc_v$,  and  $\calh_{|G_v}$, are defined in Subsection \ref{ind}). 

  \end{enumerate}
 \end{thm}

 Since we do not assume that $G$ is  one-ended relative to  $\calp\cup\calh$, there is no \emph{canonical} JSJ decomposition.

 \begin{proof}
$\bullet$
First assume that $G$ is   one-ended relative to $ \calp\cup \calh$. 
Consider the canonical elementary JSJ tree $\Tcan$ relative to $ \calp\cup\calh$ as in Subsection \ref{cano}. 

Let $G_v$ be  a vertex stabilizer containing $P_1$. 
It cannot be flexible QH because $P_1$ is   not virtually cyclic (see Remark \ref{rem_UEQH}). 
If it is rigid (non-elementary), we have seen in Subsection \ref{rela} 
that the image of  $\Out^0(G; \calp,\mk\calh)$ in $\Out(G_v)$ is finite
(recall that $\Out^0(G; \calp,\mk\calh)$ is the finite index subgroup of $\Out(G;\calp,\mk\calh)$
acting trivially on   $\Tcan/G$).
 Since $P_1$ equals its normalizer, this implies that  $\Out(P_1\fleche (G;\calp,\mk\calh))$ is finite,  a contradiction. 
Thus $G_v$ is elementary, so $G_v=P_1$.  This proves Assertion 1 in the one-ended case  (edge stabilizers of $\Tcan$ are finitely generated by Lemma \ref{fingen}).
 Assertion 2 is also clear since 
 $\Out(P_1\fleche (G;\calp,\mk\calh))$ is virtually  
 $\Out(P_1;\mk\Inc_{v},\mk\calh_{|G_v})$  by Proposition \ref{imro_rel}. 

$\bullet$
 We now consider the general case,  first assuming that $G$ is torsion-free. 
Let  $F=G_u$ be the vertex stabilizer   containing $P_1$ in a Grushko decomposition    $S$  relative to $\calp\cup\calh$ 
(see Subsection \ref{JSJ}), and let $\calp_{ | F},\calh_{|F} $ be the induced structures (see Definition \ref{indu}); 
 if $\calp\cup\calh=\{P_1\}$, then $F$ is simply the smallest free factor containing $P_1$. 
 
  Since $F$ is hyperbolic relative to $\calp_{ | F} $ by Lemma \ref{Hrufi}, 
and $\Out(P_1\fleche (G;\calp,\mk\calh))=\Out(P_1\fleche (F;\calp_{|F},\mk\calh_{|F}))$ because $F$ is a free factor (or by Remark \ref
 {automind2}),   
 the results of the previous case apply. 
The group $P_1$ is a vertex group $G_v$ of  a splitting $\Gamma_F$ of $F$, which may be used to refine   $S$ to an elementary JSJ decomposition $\Gamma$ of $G$  having $G_v$ as a vertex group 
(see  Subsection 8.1 of \cite{GL3a}). The families $\Inc_v$   and  $\calh_{|G_v}$ are the same for $\Gamma_F$ and $\Gamma$. 

$\bullet$ 
  If $G$ has torsion, we define $F=G_u$ and $\calp_{ | F},\calh_{|F} $ as above, using a  
Stallings-Dunwoody tree $S$   relative to $\calp\cup\calh$. 
The proof is technically more complicated because we cannot neglect the incident edge groups $\Inc_u$.  
  
  All Stallings-Dunwoody trees $S$ have a unique vertex stabilizer $G_{u(S)}$ equal  to $F$, but the incident edge groups may vary. This was studied in Section 4 of \cite{GL2}, where we defined a ``peripheral structure''  for $F$. 
  To state the relevant result, we choose a Stallings-Dunwoody tree $S$ for which  the valence of  $u(S)$ 
in the quotient graph of groups $S/G$ is minimal. 
Since no edge stabilizer is properly contained in a conjugate, 
it follows from Proposition 4.9 of \cite{GL2} that  the incident structure $\Inc_{u(S)}$ 
does not depend on the choice of   such an $S$ (in   trees with non-minimal valence,  there  may be more incident edge groups; 
 such a group is  contained in a group  belonging to $\Inc_{u(S)}$). 
We fix $S$, and from now on we write $u$ rather than $u(S)$, so $F=G_u$.  

Any automorphism representing an element of $\Out(G;\calp,\mk\calh)$ and leaving $P_1$ invariant 
also leaves $F$ invariant.  
Since $P_1$ and $F$ are equal to their   normalizers, $\Out(P_1\fleche (G;\calp,\mk\calh))$ is the image of the map  $p:\Out(G;\calp,\mk\calh)\to\Out(P_1)$, and $p$ factors through $\rho_u:\Out(G;\calp,\mk\calh)\to\Out(F)$.
By Remark \ref{automind2}, the image of $\rho_u$ contains $\Out(F;\mk \Inc_u,\calp_{ | F},\mk\calh_{|F})$
and is contained in $\Out(F;\calp_{ | F},\mk\calh_{|F})$.

 Our choice of $S$ implies that automorphisms in the image of $\rho_u$  preserve $\Inc_{u}$ globally. Since  
$\Inc_{u}$ consists of finitely many finite subgroups of $F$ (well-defined up to conjugacy),
the index of $\Out(F;\mk \Inc_u,\calp_{ | F},\mk\calh_{|F})$
 in $\Out(F;\calp_{ | F},\mk\calh_{|F})$ is finite. It therefore suffices to 
 study the image of 
$q:\Out(F;\mk \Inc_u,\calp_{ | F},\mk\calh_{|F})\to\Out(P_1)$, and to show that it is virtually $\Out(P_1;\mk\calk)$.

The group   $F=G_u$  is hyperbolic relative to the family $\calp_{ | F}$ (see Lemma \ref{Hrufi}), 
and one-ended relative to $\calp_{ | F}\cup\calh_{|F}$. Since the image of $q$ is infinite,  we have seen that $P_1$ is a vertex group $G_v$ in the canonical  elementary JSJ decomposition $\Gcan$ of $F$ relative to $\calp_{ | F}\cup\calh_{|F}$. One obtains an elementary JSJ tree $T$ of $G$ relative to  $\calp\cup\calh$ by refining $S$ 
using  
JSJ decompositions of vertex groups    (see Subsection 8.1 of \cite{GL3a}), so Assertion 1 is proved. 

Moreover, $\Out(P_1\fleche (G;\calp,\mk\calh))$ is virtually $\Out(P_1;\mk{\calk'})$, where $\calk'$ is the union of   $\Inc_v$ (the incident edge groups of $P_1$ in $\Tcan$) and  $(\Inc_u\cup\calh_{|F})_{ | P_1}$. 
We  now  show that $\calk'=\calk$   if we construct $T$ carefully.

   When $S$ is refined to yield $T$, the vertex $u$ is   replaced   by $\Tcan$. There is some freedom in the
way edges of $S$ containing $u$ are attached to $\Tcan$: an edge $e$ may be attached to any vertex of $\Tcan$ which is fixed by $G_e$.
We may therefore assume that,
 if $e$ is an edge of $T\setminus\Tcan$   attached to $v$, 
then $v$ is the only fixed point of $G_e$ in $\Tcan$.

The family 
$(\calh_{|F})_{ | P_1}$ is defined viewing $F$ as a vertex group of $S$, and then $P_1$ as a vertex group of $\Tcan$. Since groups in $\calh$ are infinite and edge stabilizers of $S$ are finite, $(\calh_{|F})_{ | P_1}$ equals $\calh _{ | P_1}$, defined viewing $P_1$ as a vertex group of $T$. We complete the proof by showing that $\Inc_v\cup(\Inc_u )_{ | P_1}$ is the family of incident edge groups  in $P_1=G_v$ viewed as a vertex stabilizer of $T$. 

There are two types of incident edge groups of  $G_v$ 
in $T$. Those fixing edges in $\Tcan$ are precisely those in $\Inc_v$. Because of the way we contructed $T$, those fixing edges in $T\setminus \Tcan$ have $v$ as unique fixed point in $\Tcan$, they are the groups in $(\Inc_u )_{ | P_1}$ (see Definition \ref{indu}).
\end{proof}

 If $G$ is (absolutely) hyperbolic, and $P$ is a subgroup,  
then $G$ is hyperbolic relative to $\{P\}$ 
if (and only if) $P$ is quasiconvex and almost malnormal,  
  see \cite[Theorem 7.11]{Bow_relhyp} or \cite{Osin_elementary}. 
If so,  Theorem \ref{caleg_general}   applies and describes   $\Out (P\fleche G)$, the   automorphisms of $P$ which extend to $G$.

\begin{cor}\label{cor_induit} 
Let $P$ be a   
quasiconvex, almost malnormal subgroup of a hyperbolic group $G$, with  $P\ne G$. 
\begin{itemize} 
 \item
 If $\Out (P\fleche G)$ is infinite, then $P$ is a vertex group in a  splitting of $G$
 with finitely generated edge  groups, and 
$\Out (P\fleche G)$
 is virtually 
  $\Out(P;\mk\calk)$ with $\calk$   
the family of incident edge groups
(a finite family   of finitely generated subgroups of $P$). 

\item
If   $P$ is torsion-free, then $\Out (P\fleche G)$  has a finite index subgroup with a finite classifying space.
\end{itemize} 
\end{cor}

\begin{proof}
 The first assertion follows   from Theorem \ref{caleg_general}. 
 Being quasiconvex, $P$ is a hyperbolic group.
It is proved in \cite{GL_McCool} that, if $P$ is a torsion-free hyperbolic group and $\calk$ is an arbitrary family of subgroups, then $\Out(P;\mk\calk)$   has a finite index subgroup with a finite classifying space.
\end{proof}

If $G=F_n$, every finitely generated subgroup is quasiconvex  (it is a virtual retract by \cite{Hall_1949}),
so we get:

\begin{cor} If $P\inc F_n$ is finitely generated and malnormal, 
then  $\Out (P\fleche F_n)$ is virtually $\Out(P;\mk\calk)$ for some finite family $\calk$ of finitely generated subgroups of $P$. It  has a finite index subgroup with a finite classifying space. \qed
\end{cor}

This is a partial answer to a question that was asked by D.\ Calegari. Note that the proof uses JSJ decompositions over groups which are not small.

\begin{example} Let $P\inc F_n$ be a characteristic subgroup of finite index, with $n\ge3$. 
Then $\Out (P\fleche G)$ is not virtually of the form $\Out(P;\mk\calk)$ because there exist automorphisms of $F_n$ with no nontrivial periodic conjugacy class. 
There are similar exemples with $P$ of infinite index. 
\end{example}

\section{Groups with infinitely many automorphisms}\label{outinfi}

In this section, we characterize those relatively hyperbolic groups whose automorphism group is infinite.

 In the first subsection, we point out that   determining whether $\Out(G)$ is infinite or not is relatively easy when $G$ is torsion-free or one-ended. In particular, we  give  a complete answer for toral relatively hyperbolic groups.

The most  interesting   case is thus when  $G$ has torsion  and  splits over a finite group. For instance, virtually free groups with $\Out$ finite were 
determined by M.\ Pettet \cite{Pettet_virtually}. We will give a different characterization (see Example \ref{Pett}). 

If $G$ is hyperbolic relative to $\calp$,  we will show in Subsection \ref{sec_infini_marked} that 
 the group $\Out(G;\mk \calp)$ of automorphisms which act   trivially on each parabolic subgroup is infinite if and only if $G$ has an elementary splitting relative to $\calp$ whose group of twists is infinite.

In Subsection \ref{sec_infini_unmarked}, we   get a   characterization for the full group $\Out(G;\calp)$ being infinite: 
this happens if and only if $G$ has an elementary splitting relative to $\calp$ whose group of twists is infinite,
or in which  a maximal parabolic subgroup $P$ occurs as a vertex group  and $P$ has infinitely many outer automorphisms acting trivially on incident edge groups (such automorphisms extend to $G$).

 When $G$ is hyperbolic,  we show in Subsection \ref{hyp} that $\Out(G)$ being infinite is equivalent to   $G$ having a splitting  over a  maximal virtually cyclic group with infinite center; 
this is decidable algorithmically.

\subsection{Torsion free groups}

 We first note:
\begin{lem} \label{freep}
If a torsion-free, finitely generated, group $G$ is a non-trivial free product,
then $\Out(G)$ is infinite.
\end{lem}

\begin{proof} Write $G=A*B$. If $a\in A$ is not central, the automorphism of $G$ equal to conjugation by $a$ on $A$ and to the identity on $B$ has infinite order in $\Out(G)$.  Assuming that  $\Out(G)$ is finite, we deduce that $Z(A)$ has finite index in $A$, so $A$ is abelian because $[A,A]$ is finite by a result due to Schur 
\cite[10.1.4]{Robinson_course}. Similarly, $B$ is abelian. Moreover,  $\Out(A)$ and $\Out(B)$ are finite, so $A=B=\Z$. This is a contradiction since $\Out(\Z*\Z)$ is infinite. 
\end{proof}

Thus, for torsion-free groups,  infiniteness of $\Out(G)$ is only interesting for one-ended groups.
One can get a similar result in a relative setting.
    
\begin{prop}\label{prop_freep_rel}
 Let $G$ be a finitely generated, non-cyclic, torsion-free group, and $ \calh$ a finite collection of finitely generated subgroups.
If the Grushko decomposition of $G$ relative to $ \calh$ is non-trivial, and not an amalgam  $G=A_1*A_2$ with $A_1, A_2$ abelian, 
then   $\Out(G;\mk\calh)$ is infinite.

\end{prop}

 \begin{rem} \label{pab}
  If the 
Grushko decomposition $\Gamma$  relative to $ \calh$ is    $G=A_1*A_2$ with $A_1, A_2$ abelian, then $\Out(G; \mk\calh)$ is finite if and only if, for $i=1,2$, the subgroup of $A_i$ generated by subgroups conjugate to a group in $\calh$ has finite index. This is because $\Gamma$ is $\Out(G; \mk\calh)$-invariant by  \cite{For_deformation} (its Bass-Serre tree   is the unique reduced tree 
in its deformation space). Twists are trivial because $A_1$ and $A_2$ are abelian, so $\Out(G; \mk\calh)$ is infinite if and only if $A_1$ or $A_2$ has infinitely many automorphisms acting trivially on $\calh_{|A_i}$.
\end{rem}

\begin{proof}
 
Assume that $\Out(G; \mk\calh)$ is finite, and 
let $\Gamma$ be a reduced Grushko decomposition of $G$ relative to $ \calh$. We assume that $\Gamma$ is non-trivial and we show that it is an amalgam as in the proposition.

We first note  that $G$ cannot split relative to $ \calh$ as an HNN extension $G=A*_{\{1\}}$ over the trivial
group. Indeed, 
the group of twists of this HNN extension is isomorphic to   $(A\times A)/Z(A)$,
with $Z(A)$ embedded diagonally, so contains
the infinite group $A$, a contradiction.
It follows that $\Gamma$ is a tree of groups.

The proof of Lemma \ref{freep} shows that, whenever $G$ splits as a free product $A*B$ relative to $ \calh$,
then $A$ and $B$ are abelian: otherwise the group of twists of the splitting is infinite.
Since $\Gamma$ is reduced, it follows that it is an amalgam $G=A_1*A_2$ with $A_1, A_2$ abelian:  if $\Gamma$ has more than one edge, collapsing an edge   provides a decomposition 
with a non-abelian vertex group.
\end{proof}

 Let  now $G$ be hyperbolic relative to  
  $\calp=\{P_1,\dots,P_n\}$,  with $P_i$    finitely
    generated, not virtually cyclic. 

If $G$ is one-ended relative to $\calp$,     one can   read infiniteness of
$\Out(G;\calp)$ from 
the JSJ decomposition thanks to
Theorem \ref{thm_struct_m}:   
$\Out(G;\calp)$ is finite if and only if the canonical elementary JSJ decomposition relative to $\calp$ has no flexible  QH vertex  with infinite mapping class group, the parabolic subgroups $P_j$ appearing as vertex stabilizers have 
$\Out(P_j;\mk\Inc_{P_j})$ finite, and the group of twists is finite. 
 We may be more specific under additional conditions on the parabolic subgroups.

 \begin{prop} \label{rht} Let $G$ be a non-abelian  toral relatively hyperbolic group. The following are equivalent:
 \begin{enumerate}
 \item  $\Out(G)$ is finite.
 \item Every non-trivial abelian one-edge splitting of $G$   is an amalgam $A*_CB$ with $C$ of finite index in $A$ or $B$.
 \item   $G$ has no non-trivial splitting over an abelian subgroup stable under taking roots.
 \item    $G$ is freely indecomposable and its canonical abelian JSJ decomposition $\Gcan$ relative to non-cyclic abelian subgroups satisfies the following:
 \begin{itemize*}
 \item $\Gcan$ consists of a central   vertex, possibly  connected to terminal vertices carrying an abelian group;
 \item   the central vertex is  rigid, or QH with underlying surface $\Sigma$ homeomorphic to a pair of pants or a twice punctured projective plane;
 \item at each terminal vertex, the incident edge group has finite index in the vertex group. 
 \end{itemize*}
 \end{enumerate}
 \end{prop}

 A subgroup $A$ is stable under taking roots if  $g\in A$  whenever $g^n\in A$ for some $n\ge2$   (this is also called pure, or isolated).   

 The pair of pants and the twice punctured projective plane appear in this statement because they are the only compact hyperbolic surfaces with finite mapping class group (see the end of Subsection \ref{autar}).  The fundamental group of  a pair of pants is rigid. 
The fundamental group of a twice punctured projective plane has   two (incompatible) cyclic splittings
relative to the boundary 
(it is flexible), but none over a maximal cyclic subgroup. 

 Automorphisms of toral relatively hyperbolic groups were considered in \cite{DaGr_isomorphism}, and some of the equivalences in Proposition \ref{rht} 
  follow from their results  (note for instance that a splitting as in (3) is an essential splitting in the sense of their Definition 3.30).

\begin{proof}
 If $G$ is a free product, (2) and (3) are false, and so is (1) by 
  Lemma  \ref{freep}. We  therefore assume that   $G$ is freely indecomposable. 

We prove $(1)\imp (2)$ by assuming that (2) does not hold, and  showing that $\Out(G)$ is infinite. 
If $G$ is an HNN extension over an abelian group, or an amalgam with $A$ and $B$ non-abelian, the group of twists of the splitting is infinite. 
If $G=A*_CB$ with $A$ abelian   containing $C$ with infinite index, $\Out(G)$ is infinite because $A$ has nontrivial automorphisms equal to the identity on $C$. 

It is clear that (2) implies (3).  To prove that (3) implies (2), first   assume that $G=A*_C B$ 
with $C$ abelian of infinite index in both $A$ and $B$.
Let $\Hat C$ be the set of all roots of elements of $C$, an abelian subgroup containing $C$ with finite index    (recall that all abelian subgroups of $G$ are cyclic or parabolic, hence finitely generated).
Since $\hat C$ is elliptic  in the amalgam, up to exchanging the role of $A$ and $B$, we can assume that $\Hat C<A$.
Then $G=A*_{\hat C}\grp{B,\hat C}$ is a decomposition contradicting (3).
The case of an HNN extension is similar, but we do not need the hypothesis that $C$ has infinite index in $A$.

If (4) does not hold,  we construct a splitting contradicting (2).
If $\Gcan $  has a flexible QH vertex, and if the underlying surface is not a twice punctured projective plane,
then
one   
simply considers the cyclic splitting dual to a  2-sided  essential simple closed curve not bounding a M\"obius band.
The other possibility is that $\Gcan $  
  has a vertex $v$  carrying an  abelian group such that either $v$ has   valence $\ge2$,   or $v$ is terminal with  the edge group of infinite index in $G_v$.  One gets the required splitting by   collapsing edges of $\Gcan$. 

 We have proved   $(1)\imp (2)\imp (4)$ and $(2)\Leftrightarrow(3)$ .
We conclude by  deducing from   the exact sequence of Corollary \ref{limgp} that $\Out(G)$ is finite if (4) holds. 
  The groups  $GL_{r_k,n_k}(\Z)$ are trivial because $r_k=0$ by the finiteness assumption at terminal vertices. 
 The mapping class group of $\Sigma$ is finite. 
 Twists are trivial because terminal vertices carry an abelian group.
\end{proof}

 \begin{prop}  Let $G$ be torsion-free, hyperbolic relative to nilpotent subgroups. Assume that $G$ is not nilpotent. If $\Out(G)$ is finite, then    $G$ is freely indecomposable and
  its canonical   JSJ decomposition $\Gcan$ over nilpotent groups relative to non-cyclic nilpotent  subgroups consists of a central vertex
which is 
 rigid, or QH with underlying surface $\Sigma$ homeomorphic to a pair of pants or a twice punctured projective plane,
possibly  connected to terminal vertices carrying a nilpotent group.   
 \end{prop}

\begin{proof}  
 We may assume that  no $P_i$ is cyclic, so   $\Out(G)=\Out(G;\calp)$. As above, $G$ is freely indecomposable and there is no    QH vertex
other than  those mentioned. 
Recall that an infinite torsion-free nilpotent group has infinite center.  As in the proof of Corollary \ref{limgp}, the group of twists of $\Gcan$ contains the direct product $\prod (Z(G_v))^{ | E_v | -1}$ taken over vertices carrying a nilpotent  group, so is infinite as soon as there is a vertex with valence $ | E_v |\ge2$. 
\end{proof}

\subsection{Infinity of marked automorphisms}\label{sec_infini_marked}

\begin{thm}\label{thm_marked}
  Let $G$ be hyperbolic relative to $\calp=\{P_1,\dots, P_n\}$, with each $P_i$ 
  finitely generated.
 Then $\Out(G;\mk\calp)$ is infinite if and only if 
there is   a splitting of $G$ over  finitely generated elementary subgroups,  relative to $\calp$,
 with an infinite group of   twists  $\calt$ (see Subsection \ref{arb}).
 
 More generally,   if $\calq$ is a finite family of finitely generated subgroups  with $\calp\inc\calq$, 
 the same characterization holds for $\Out(G; \mk\calq)$, with a splitting relative to $ \calq$. 
\end{thm}

   By Lemma \ref{lem_collapse_twist}, the splittings may be assumed to
  have only one edge.   The proof will be given  at the end of the subsection. 

 \begin{example} \label{Pett}
Consider the virtually free group  
$G= D_4*_CD_6 *_C D_4$,
where $D_n$ is the dihedral group of order $n$ and $C$ has order 2  (note that $C$ is  central in $D_4$, but equal to its centralizer in $D_6$). For this  two-edge splitting $\Gamma$,
and any splitting obtained by collapsing an edge, the group of twists is finite.
 The one-edge splitting $\Delta$ given by the amalgam $G=D_6*_C \left(D_4*_C D_4\right)$, however,
has a   twist of infinite order. 
  The Bass-Serre tree of $\Delta$ is the tree of cylinders of the  Bass-Serre  tree $T$ of $\Gamma$, and 
Assertion 1 of Proposition \ref{prop_induct} holds in this case. Compare \cite{Pettet_virtually}.
\end{example}

When $G$ is hyperbolic,  the group of twists of the splitting provided by Theorem \ref{thm_marked} contains an element of infinite order 
(see also  Theorem \ref{thm_twist_hyp}).  
The following example shows that this does not hold for general relatively hyperbolic groups.

 \begin{example}\label{ex2}
  Let $G=B_1*B_2$ be the free product of 
 two infinite torsion groups with trivial center.
 It is hyperbolic relative to $\calp=\{B_1, B_2\}$, and $\Out(G;\mk\calp)$ is infinite. But no splitting over elementary subgroups has a twist of infinite order, as we now show. By way of contradiction, suppose that some $Z_{G_{o(e)}}(G_e)$ contains an element   of infinite order. 
   The group $G_e$ is trivial, or  contains a torsion element $g\ne1$, or is isomorphic to $\Z$. It cannot be trivial since $G_{o(e)}$ would then be a torsion group. The existence of  a torsion element 
 $g$ also leads to a contradiction   since the centralizer of  such a $g$ is a torsion group. We conclude by observing that $G$ 
  does not split over $\Z$: if it does,  an equivariant  map from the Bass-Serre tree of the amalgam $B_1*B_2$ to that of the splitting maps vertices to vertices and must be locally injective, hence globally injective, a contradiction.
   \end{example}

Before proving
Theorem \ref{thm_marked}, we note the following fact, which follows from the presentation of $\Tw$ recalled in Subsection \ref{arb}.

\begin{lem} \label{ti} Let $G$ be a relatively hyperbolic group. Let $\Gamma$ be
  a one-edge  splitting  of $G$ over   a virtually cyclic group $G_e$ with infinite center, with $G_e$ not parabolic. 
Any element of  infinite order   in $Z(G_e)$ defines a   twist which  has infinite order in $\Out(G)$, unless $\Gamma$ is an amalgam and one of the vertex groups is virtually cyclic with infinite center. \qed
\end{lem}

The following result is a key step in the proof of Theorem \ref{thm_marked}.

\begin{prop}\label{prop_induct}
  Let $  T$ be a  non-trivial tree with edge stabilizers of constant finite cardinality $k$.
Let $T_c$ be its tree of cylinders for the equality equivalence relation (see Subsection \ref{defcyl}).
Then at least one of the following holds:
\begin{enumerate}
\item $T_c$ is nontrivial and its edge stabilizers   are finite;
\item $T$ has a collapse $T'$ which is nontrivial and  has an infinite group of twists $\Tw(T')$;
\item  $T$ has a collapse $T'$ which is nontrivial and invariant under $\Out(\cald(T))$.\end{enumerate}
\end{prop}

$\cald(T)$ denotes the   deformation space of $T$ over groups of cardinality   $\le k$. All reduced trees in $\cald(T)$ have edge stabilizers of order $k$  (see Subsection \ref{defsp}). Also note that $T_c$ is invariant under $\Out(\cald(T))$,   and dominated by $T$, so we get:

\begin{cor} \label{propc_induct}
$T$ has a collapse   which is nontrivial and  has an infinite group of twists, or   $T$ dominates a nontrivial tree $T'$ with finite edge stabilizers which is invariant under $\Out(\cald(T))$. \qed
\end{cor}

\begin{proof}[Proof of  Proposition \ref{prop_induct}]
We can assume that $T$ is reduced.  
  We first consider the case when $T_c$ is trivial. 
Since edges of $T$ belong to the same cylinder if and only if they have the same stabilizer,   all edges of $T$ have the same stabilizer, a finite normal subgroup $A$. 

If there is only one orbit of edges, $T $ is the unique reduced tree in $\cald(T)$  because no edge stabilizer may be properly contained in another 
 \cite{Lev_rigid}.
It follows that $T'=  T$ is invariant under $\Out(\cald(T))$.

Assume that there is more than one orbit of edges. 
 Since $A$ is normal  and finite, its centralizer in $G$ has finite index, so   for any vertex $v$ in an arbitrary  collapse  $T'$ of $T$ (including $T$ itself),
and any edge $e$  of $T'$  incident to $v$,  the group $Z_{G_v}(G_e)$ is infinite as soon as $G_v$ is infinite.
By Lemma \ref{lem_twist0},  if Assertion 2 does not hold, then all vertex stabilizers of  nontrivial collapses $T'$
are either  finite or infinite with infinite center. 
Since $T$ is reduced, this is possible only if $T$ has  one orbit of vertices,
all vertex stabilizers equal to $A$,
and only two orbits of edges 
  (this is respectively because a non-trivial amalgam $H_1*_A H_2$,
an HNN extension  $H_1*_A$ with $H_1\neq A$, and 
a double HNN extension $(A*_A)*_A$, have finite center).
In particular, $G/A$ is free of rank $2$. 
One easily checks that collapsing the orbit of any edge gives a tree $T'$ with  $\Tw (T')$ infinite, 
though Lemma \ref{lem_twist0} does not apply.
This concludes the case when $T_c$ is trivial.

From now on, we assume that $T_c$ is nontrivial and some edge $\eps=(x,Y)$ of $T_c$ has infinite stabilizer.
View the cylinder $Y$ as a subtree of $T$ containing $x$, and consider an edge $e\inc Y$ with origin $x$. Then  $G_Y=N_G(G_e)$ and $G_\eps=G_Y\cap G_x=N_{G_x}(G_e)$. We know that $Z_{G_x}(G_e)$ is infinite, but we cannot apply Lemma \ref{lem_twist0} since we do not know that $Z(G_x)$ is finite.

Assume for a moment that $Y$ contains edges from at least two $G$-orbits. Consider $  T'$ obtained from $T$ by collapsing all edges in the orbit of $e$,
and denote by $  x'$ the image of $x$  in $T'$. Since $T$ is reduced, $G_x\subsetneq G_{ x'}$.
By the assumption on $Y$, there is an edge $e'$ of $  T'$ incident to $  x'$ with $G_{e'}=G_e$.
Since $Z_{G_{ x'}}(G_{e'})$ is infinite, we are done if  $Z(G_{ {x'}})$ is finite. We now show that  $Z(G_{ {x'}})$ being infinite leads to a contradiction. 
Edge stabilizers of $T$ being finite, $x$ is the unique point of $T$ fixed by $G_\eps$. 
 Since $G_\varepsilon\inc G_x\inc G_{ x'}$,  the group $Z(G_{ {x'}})$ normalizes $G_\eps$, hence also fixes $x$, 
and only $x$ since it is infinite.
This implies that $G_{ x'}$ fixes $x$, a contradiction to $G_x\subsetneq G_{ x'}$.

Returning to the general case, there is a $G$-invariant partition of the set of cylinders: those for which there is an edge $(x,Y)$ of $T_c$ with infinite   stabilizer, and the others. Thanks to the previous argument, we may   assume that all edges contained in a given  cylinder  of the first type belong to the same $G$-orbit. 
Let now $T'$ be the tree obtained from $T$ by collapsing all edges in  cylinders of the second type. It is nontrivial (but $ T'=T$ is possible). 
We show that $T'$ does  not change if we replace $T$ by another reduced tree $T_1$ in $\cald(T)$. This implies that $T'$ is invariant under $\Out(\cald(T))$.

One may join $T$ and $T_1$    by slide moves (see Subsection \ref{defsp}).  In a slide move, an edge $e$ slides over an edge $f$   belonging to a different orbit, with $G_e\inc G_f$. Here one must have $G_e=G_f$, so $e$ and $f$   belong to the same cylinder, necessarily of the second type. The slide move does not change $T'$ since the cylinder gets collapsed. 
\end{proof}

\begin{proof}[Proof of Theorem \ref{thm_marked}]
 We prove the ``only if '' direction (the other direction is clear since twists act by conjugations on vertex groups).
We may assume that all groups in $  \calq$ are infinite.  
We first suppose that $G$ is one-ended relative to $ \calq$.
Let $T$ be the canonical  elementary JSJ tree relative to  $ \calq=\calp\cup\calh$ as in Subsection \ref{cano}.  Its edge stabilizers are finitely generated by
Lemma \ref{fingen}. 
By Theorem 
\ref{thm_struct_m_rel},
if $\Out(G; \mk\calq)$ is infinite, then either $\Tw(T)$ is infinite and we are done,
or $T$ has at least a non-rigid QH vertex $v$ with  $\Out(G_v;\calb_v)$ infinite.

The underlying orbifold  group  $\pi_1(\Sigma)$ splits, relative to its boundary subgroups, over a maximal  virtually cyclic subgroup with infinite center  (see  \cite{DG2}, Proposition  3.1).   This induces an elementary  splitting of $G_v$ which extends to a splitting of $G$. 
By Lemma  \ref {extens}, this splitting of $G$ is relative to $ \calq$ (because the intersection of $G_v$ with a conjugate of a group in $ \calq$  projects into a boundary subgroup in $\pi_1(\Sigma)$), and  has an infinite group of   twists  by Lemmas \ref{ti} and \ref{lem_extension_twist}.
 This   proves the theorem if $G$ is one-ended relative to $\calq$. 

In the general case, 
let $k\in \bbN$ be the smallest number such that $G$ splits relative to $ \calq$ over a group of cardinality $k$.
Let $T$ be a reduced JSJ tree  relative to $ \calq$ over subgroups of cardinality $k$ (see Subsection \ref{JSJ}). 
Its deformation space   is invariant under $\Out(G; \mk\calq)$.
By Corollary \ref{propc_induct}, either some collapse of $T$ has an infinite group of   twists (and we are done), 
or $T$ dominates  a nontrivial  tree $T'$ with finite edge stabilizers 
  which is invariant under  
$  \Out(G;\mk\calq)$.
Note that the groups  in $ \calq$ are elliptic in $T'$, since they are elliptic in $T$ and $T$ dominates $T'$.

We may   assume that $\Tw(T')$ is finite. Let $ A_0\inc \Out(G; \mk\calq)$ be the finite index subgroup  acting trivially on the graph $T'/G$.
 By  Assertion 1 of Lemma \ref{lem_virt_nobitwist}, there exists a vertex group $G_v$ of $T'$ such that 
$\rho_v(A_0)\subset \Out(G_v)$ is infinite. In particular, $G_v$ is infinite.

The group $G_v$ is   hyperbolic relative to the family  $\calp_{|G_v}$ (see Lemma \ref{Hrufi}).
By Lemma \ref{automind}, 
we have $\rho_v(A_0)\inc \Out(G_v; \mk\calq_{|G_v})$.

 If  the theorem holds for $G_v$,
we get a graph of groups decomposition $\Gamma_0$ of $G_v$ relative to $ \calq_{|G_v}$ having an infinite group of   twists. Since $T'$ has finite edge stabilizers, $\Gamma_0$ is relative to $\Inc_v$ and one may refine $T'/G$
 to 
  a graph of groups $\Lambda$  by using $\Gamma_0$.
By Lemma \ref{lem_extension_twist}, the splitting  $\Lambda$ has an infinite group of twists. It is relative to $ \calq$ by Lemma \ref{extens}. 

 If  the theorem does not hold for $G_v$, we repeat the construction. If the process stops after finitely many steps, we get a splitting $\Lambda$ as in the previous case. If it does not stop, we get an infinite sequence of trees $T_i$ relative to $\calq$ with finite edge stabilizers, with $T_{i+1}$ strictly dominating $T_i$.  Since $G$ is finitely presented relative to $\calq$, there is  a Stallings-Dunwoody decomposition  relative to  $\calq$ (see Subsection \ref{JSJ}), and we reach a contradiction (see \cite[p.\ 130]{DicksDunwoody_groups}).
\end{proof}

\subsection{Infinity of unmarked automorphisms}  \label{sec_infini_unmarked}

 Using Theorem \ref{caleg_general}, we now
characterize relatively hyperbolic groups for which $\Out(G;\calp)$
 is infinite.
 
We first note the following consequence of Theorem \ref{thm_marked}.

\begin{cor} \label{cor_ker} 
 Let  $G$ be  hyperbolic relative to  $\calp=\{P_1,\dots,P_n\}$, where each $P_i$ is infinite and finitely generated, and let $\calh$ be a finite family of finitely generated subgroups.

If $\Outu(G;\calp,\mk\calh )$ is infinite, there is an elementary splitting relative to $\calp\cup\calh$ 
with an infinite group of twists, or there is an $i$ such that the natural map $\Outu(G;\calp,\mk\calh )\to\Outu(P_i)$ (defined because $P_i$ equals its normalizer) has infinite image. 
 \end{cor}

 \begin{proof}
   If all maps $\Outu(G;\calp,\mk\calh )\to\Out(P_i)$ have finite image, then
the intersection of their kernels, namely $\Out(G;\mk\calp,\mk\calh)$, is infinite.
We can now apply Theorem \ref{thm_marked}.
 \end{proof}

 \begin{cor}
   \label{cor_outu_infini}
 Assume furthermore that $G$ is non-elementary (\ie 
$P_i\ne G$).
Then 
  $\Out(G;\calp,\mk\calh)$ is infinite if and only if  $G$ has an  
  elementary splitting as a graph of groups $\Lambda$ relative to $ \calp\cup\calh$ 
such  that one of the following holds:
 \begin{itemize}
\item 
  the group of   twists of $\Lambda$ is infinite, 
\item
 or $\Lambda$ has a vertex $v$ such that $G_v=P_i$ is a maximal parabolic subgroup  and
$\Out(G_v;\mk\Inc_v,\mk\calh_{|G_v})  $ is infinite.
\end{itemize}
 \end{cor}

 \begin{proof}   As in Theorem \ref{thm_marked},   the ``if'' direction is clear, so assume that $\Out(G;\calp,\mk\calh)$ is infinite.  
   By Corollary \ref{cor_ker}, and up to renumbering, we can assume that 
 $\Out(P_1\fleche (G;\calp,\mk\calh))$ is infinite.
By Theorem \ref{caleg_general}, $P_1$ is a vertex group in some  elementary decomposition 
of $G$ relative to $\calp\cup\calh$ 
 and $\Out(P_1;\mk\Inc_v,\mk\calh_{|G_v})$ has finite index in 
 $\Out(P_1\fleche (G;\calp,\mk\calh))$.
In particular, $\Out(P_1;\mk\Inc_v,\mk\calh_{|G_v})$ is infinite.
 \end{proof}

\subsection{Hyperbolic groups} \label{hyp}

We apply the results of the previous subsections to the case when $G$ is a hyperbolic group. 
We say that a 
 subgroup of $G$ is $\Zmax$ if it is  maximal for inclusion among virtually cyclic subgroups
with infinite center.  Note that any virtually cyclic subgroup $C$ with infinite center is contained in a unique $\Zmax$ subgroup $\hat C$ (the pointwise stabilizer of $\partial C$).

 Given any splitting of a hyperbolic group over a $\Zmax$ subgroup $C$, any central element $c\in C$ of infinite order defines
a twist of infinite order in $\Out(G)$.

\begin{thm}\label{thm_twist_hyp} 
  If $G$ is hyperbolic,   $\Out(G)$ is infinite if and only if there is  a  non-trivial splitting  of  $G$  
 over a $\Zmax$  subgroup  (such a splitting always has a    twist of infinite order). 

If  $\calh$ is a finite family of finitely generated subgroups,  $\Out (G;\mk\calh)$ is infinite if and only if there is a non-trivial
splitting   over a $\Zmax$ subgroup
relative to $\calh$.
\end{thm}

It  was proved independently by M.\ Carette \cite{Carette_automorphism}  that $\Out(G)$ is infinite if and only if
$G$ has a splitting over 
 a finite group or a (maybe non maximal) 
virtually cyclic group with infinite center,
with a twist of infinite order 
  (see \cite{Lev_automorphisms,DG2} for the one-ended case).

\begin{proof}   The ``if'' direction is clear, so we assume that   $\Out(G;\mk\calh)$ is infinite. All splittings considered in this proof will be relative to $\calh$. 

Theorem \ref{thm_marked} and Lemma \ref{lem_collapse_twist} say that $G$ has a one-edge splitting over a  (possibly finite) virtually cyclic group $C$, whose group of twists is infinite.   We assume that this splitting is an amalgam $G=A*_C B$; the case of an HNN extension is similar.
We first explain  how to replace this amalgam over $C$ by one over a (possibly non-maximal) virtually cyclic  subgroup  $C'$ with infinite center.

 If $C$ is infinite with finite center,   its centralizer in $G$ is finite  and this forces the group of twists to be finite.
So assume that $C$ is finite.  

  If both $Z_A(C)/Z(A)$ and $Z_B(C)/Z(B)$ are finite, the group of twists is finite, so assume for instance that $Z_A(C)/Z(A)$ is infinite. 
Note that $Z(A)$ has to be finite, since otherwise  $A$ would be  virtually cyclic, 
and $Z_A(C)/Z(A)$ would be finite.
Consider an element of infinite order $t\in Z_A(C)$,  and perform a fold to get  $G=A*_{\grp{C,t}} \grp{B,t}$. This is a splitting over a virtually cyclic  subgroup  $C'$ with infinite center, and it is relative to $\calh$.
The twist defined by $t$
  has infinite order in $\Out(G)$   by Lemma \ref{ti}. 

Now suppose that an amalgam $G=A'*_{C'} B'$ has an infinite group of twists, with $C'$ virtually cyclic     with infinite center.
Then $A',B'$ have finite center. The $\Zmax$ subgroup $\Hat C'$ containing $C'$
is elliptic in the amalgam, and one can perform a fold to get  an amalgam  over $\Hat C'$. This splitting is non-trivial because $\Hat C'$ 
is not conjugate to  $A'$ or $B'$ since they have finite center. The group of twists of the new splitting is clearly infinite.
\end{proof}

\begin{thm} \label{thm_MC_infini}
  There is an algorithm which, given a hyperbolic group $G$, decides whether $\Out(G)$ is infinite  or not. More generally,  if $\calh$ is a finite family of finitely generated subgroups, one may   decide whether 
  $\Out(G; \mk\calh)$ is infinite. 
\end{thm}

\begin{proof}  We start with the first assertion. 

We first construct an algorithm that stops if and only if $\Out(G)$ is finite.
  By Theorem 8.1 of \cite{DG2}, one can compute a finite generating set of $\Out(G)$.
Moreover, one can solve the word problem in $\Out(G)$ as this amounts to solving  (uniformly)
the simultaneous
conjugacy problem in $G$ (see \cite{BriHow_conjugacy} for a solution).   
Thus, for each $R>0$, one can  determine  the ball $B_R$ of radius $R$ in the Cayley graph of $\Out(G)$. Checking whether $B_R=B_{R+1}$ for some $R$ gives  the required algorithm.

It now suffices to construct an algorithm that stops if $\Out(G)$ is infinite. By \cite[Lemma 2.8]{DG2}, 
one can decide whether a   subgroup of $G$ (given by generators)   is $\Zmax$ or not. 
One can therefore 
enumerate all decompositions of $G$ as an amalgam or HNN extension over
$\Zmax$ subgroups.
By Corollary \ref{thm_twist_hyp}, this provides an algorithm that stops if $\Out(G)$ is infinite.

 The argument to decide whether 
  $\Out(G; \mk\calh)$ is infinite is similar. 
 The first algorithm is the same since Theorem 8.1 of \cite{DG2} provides generators for $\Out(G;\mk\calh)$. 
For the second algorithm, one has to restrict to splittings relative to $\calh$, so one needs an algorithm that, given a splitting, stops if the splitting is relative to $\calh$. This is done by choosing a generating set $S_i$ for each $H_i\in\calh$, enumerating all conjugates of $S_i$, and comparing them with words written using the generators of a vertex group.
\end{proof}

  In general, we do not know how to decide whether 
  $\Out(G; \calh)$ is infinite  (see Remark \ref{rips}). The following is an   answer  when $G$ is hyperbolic relative to $\calh$.

\begin{prop} There is an algorithm which, given 
a torsion-free hyperbolic group $G$, a finite family $\calp$ 
of finitely generated   locally quasiconvex subgroups   $P_i$ such that $G$ is hyperbolic relative to $\calp$,
and a finite family $\calh$ of finitely generated subgroups,
decides whether $\Out(G;\calp,\mk\calh)$ is infinite.
 \end{prop}

 Since $G$ is assumed to be hyperbolic relative to $\calp$, each $P_i$ is quasiconvex in $G$.
In particular, $P_i$ is itself a hyperbolic group. Local quasiconvexity of $P_i$ means that its finitely generated subgroups are quasiconvex (in $P_i$, hence also in the hyperbolic group $G$).

 \begin{proof}
First, using Touikan's algorithm \cite[Theorem A]{Touikan_finding},
one can decide whether $G$ splits as a free product relative to $\calp\cup\calh$.
If it does, it is easy to decide whether  $\Out(G;\calp,\mk\calh)$  is infinite using Proposition \ref{prop_freep_rel} and Remark \ref{pab}.

So assume that $G$ is one-ended relative to $\calp\cup\calh$.  We may also assume that no $P_i$ is cyclic. 
We use \cite[Theorem C]{Touikan_finding} to decide whether $G$  splits in a suitable way.
For this, we need our parabolic groups $P_i$ to be algorithmically tractable 
in the sense of \cite[Definition 1.13]{Touikan_finding}. 

Since $P_i$ is locally quasiconvex, it is hyperbolic and the conjugacy problem is solvable in $P_i$.
Moreover, local quasiconvexity  of  $P_i$ implies that 
one can decide whether a finite  subset $S\subset P_i$ generates $P_i$  
or not, by checking whether a given generating set of $P_i$ lies in the quasiconvex subgroup $\grp{S}$ \cite{Kapovich_detecting}. 
This says that $P_i$ is algorithmically tractable. 

Applying \cite[Theorem C]{Touikan_finding}, one can decide whether there exists 
an elementary splitting of $G$ (viewed as a relatively hyperbolic group) relative
to $\calp\cup\calh$ with finitely generated edge groups, and if so find one. 
By local  quasiconvexity of  $P_i$, edge groups of the splitting are quasiconvex in the hyperbolic group $G$, 
and so are vertex groups (see Subsection \ref{qconv}).

Iterating this process, one can compute a maximal elementary splitting $\Gamma$ of $G$ relative
to $\calp\cup\calh$ (\ie a splitting that cannot be refined non-trivially into an elementary
splitting relative to $\calp\cup\calh$).
 Arguing as in \cite[Section 6]{DaGr_isomorphism}
and \cite[Lemma 2.34]{DG2}, one may then   
recognize the QH subgroups in $\Gamma$, and 
  find the canonical elementary JSJ decomposition of $G$ relative to $\calp\cup\calh$
  (see also \cite[Theorem 3.12]{DaTo_isomorphism}).

By Theorem \ref{thm_struct_m_rel},  $\Out(G;\calp,\mk\calh)$ is infinite if and only if
  the group of twists $\calt$ is infinite, or there is a vertex $v$ such that   $\Out(G_v;\mk\Inc_v,\mk\calh_{|G_v})$ is infinite.
  
Recall that $\calt$ is isomorphic to a quotient of $\prod_{e\in E} Z_{G_{o(e)}}(G_e)$ by edge and vertex relations.
Each group $Z_{G_{o(e)}}(G_e)$ is either trivial or infinite cyclic, and is computable.
Edge relations and vertex relations are generated by embeddings of groups $Z(G_e)$ and $Z(G_v)$ in this product.
Since one can compute the corresponding subgroups of the abelian group $\prod_{e\in E} Z_{G_{o(e)}}(G_e)$,
one can decide whether the group of twists is infinite or not.

To decide whether $\Out(G_v;\mk\Inc_v,\mk\calh_{|G_v})$ is infinite, 
  we can apply   \cite{DaGr_isomorphism} (or more explicitly  \cite[Corollary 3.4]{DG2}), 
or Theorem \ref{thm_MC_infini}, 
but we  need to determine $\calh_{|G_v}$ (see Definition \ref{indu}).
Given $H\in \calh$ generated by a finite set $S$, one can decide whether there is $g\in G$ such that $S^g\subset G_v$
in the same way as above because $G_v$ is quasiconvex. 
One can similarly decide whether $H$ fixes an edge, which allows to compute $\calh_{|G_v}$.
 \end{proof}

\section{Fixed subgroups} \label{fixed}

 In this section, we use   JSJ decompositions to study fixed   subgroups  of automorphisms.
This is     inspired by  arguments due to  Sela \cite{Sela_Nielsen}. The proof of the Scott conjecture (Theorem \ref{scott}) given below is not  really new, but   using the relative JSJ decomposition makes  the argument more direct. Using Theorem \ref{somrig}, we will prove in \cite{GL_McCool} that, given a toral relatively hyperbolic group $G$, there are only finitely many possibilities for fixed subgroups of automorphisms of $G$, up to isomorphism. This was proved by   Shor \cite{Shor_Scott} for $G$ torsion-free hyperbolic. 

\begin{thm}[\cite{BH_tt}]\label{scott} Let $\alpha$ be an automorphism of a  free group $F_n$. Its fixed subgroup $\Fix\, \alpha$ has rank at most $n$. 
\end{thm}

\begin{proof} The smallest free factor containing $\Fix\, \alpha$ is $\alpha$-invariant.   Replacing $G$ by this free factor,   we may assume that $F_n$ is one-ended (freely indecomposable) relative to $\Fix\, \alpha$. We also assume that $\Fix\, \alpha$ is not cyclic. 

Let $T$ be the canonical cyclic JSJ tree relative to $\Fix\, \alpha$ (see Subsection \ref{cano}), and let $G_v$ be the vertex stabilizer containing $\Fix\, \alpha$. It is $\alpha$-invariant because $T$ is invariant and $v$ is the only vertex fixed by $\Fix\alpha$, and abelianizing shows   that it  has rank $\le n$. 
By Remark \ref{rem_UEQH} it cannot be flexible QH because $\Fix\, \alpha$ is not cyclic, so it is rigid. By standard arguments due to Paulin and Rips, 
 $\alpha$ has finite order in $\Out(G_v)$: otherwise, applying Theorem \ref{thm_sci} with $\calp $ consisting of incident edge groups and $\calh$ consisting of $\Fix\, \alpha$ (which is finitely generated) yields a cyclic splitting of $G_v$ which contradicts rigidity.
  By Dyer-Scott \cite{DySc_periodic}, $\Fix\, \alpha$ is a free factor of $G_v$ so has rank $\le n$. 
 
 If we do not wish to use Gersten's result that $\Fix\, \alpha$ is finitely generated  \cite{Gersten_fixed_adv}, 
we argue  by contradiction as follows. Let $H$ be a finitely generated free factor of $\Fix\, \alpha$ of rank $>n$. 
  We claim  that $F_n$ is  one-ended 
  relative to $H$ (see \cite{Perin_elementary}
 and Lemma 7.6 of \cite{GL3b} for more general statements).
Otherwise, let $\Hat H$ be the smallest free factor of $F_n$ containing $H$.
Then $\Hat H$ is $\alpha$-invariant, and $\Fix\alpha\cap\Hat H$ has rank at most $n-1$ 
(assuming that the theorem holds in $\Hat H$ by induction on $n$), a contradiction since 
$\Fix\alpha\cap\Hat H$ retracts onto $H$.

Define $G_v$ as above, using the cyclic JSJ splitting   of $F_n$ relative to $H$. The fixed subgroup of $\alpha_{ | G_v}$ is a subgroup of $\Fix\, \alpha$ which has rank $\le n$ and contains $H$. This is a contradiction since $H$ is a retract of $\Fix\, \alpha$. 
\end{proof}

  This proof uses \cite{DySc_periodic}, which is specific to free groups. Applying the same argument to  (relatively) hyperbolic groups only yields:

\begin{thm} \label{somrig}
  Let $G$ be hyperbolic relative to    a    finite family  $\calp$  of slender subgroups. 
 Consider  $\alpha\in \Aut(G;\calp)$
such that $\Fix\alpha$ is not elementary (\ie not 
virtually cyclic or parabolic).
Then   $\Fix\alpha$ is contained in  an   $\alpha$-invariant  vertex  group $G_v $ of a splitting of $G$ over elementary subgroups relative
to $\calp$,    and   $\alpha_{|G_v}$ has finite order in $\Out(G_v)$.
\end{thm}

\begin{proof} 
 The proof is similar to the one above. Note that $\Fix\alpha$ is finitely generated  by \cite[Cor.\ 9.2]{Hruska_quasiconvexity}   
because it is relatively quasiconvex \cite{MiOs_fixed} and groups in $\calp$ are slender.  
First assume that $G$ is one-ended relative to $\calp\cup\{\Fix\alpha\}$. 
    Let $\Tcan$ be the canonical elementary JSJ decomposition
of $G$ relative to $\calp\cup\{\Fix\alpha\}$. 
Let $v$ be the unique  vertex of $\Tcan$ fixed by $\Fix\alpha$. As above, $G_v$ is $\alpha$-invariant; 
  it  cannot be  QH   because it contains the universally elliptic subgroup $\Fix\alpha$
which is not virtually cyclic,
so $G_v$  is rigid: it has no elementary splitting relative to 
$\Inc_v\cup\calp_{|G_v}\cup\{\Fix\alpha\}$.
By Lemma \ref{rh}, $G_v$ is hyperbolic relative to $\Inc_v\cup\calp_{|G_v}$.
By Theorem \ref{thm_sci}, the group  $\Out(G_v;\Inc_v,\calp_{|G_v}, \mk{\{\Fix \alpha\}})$ is finite. 
It contains the class of $\alpha_{|G_v}$ by Lemma \ref{automind}, so 
  $\alpha_{|G_v}$ has finite order in $\Out(G_v)$.

If $G$ is not relatively one-ended, we consider   a reduced Stallings-Dunwoody tree   $S$ relative to $\calp\cup\{\Fix\alpha\}$ (see Subsection \ref{JSJ}).
Since $\Fix\alpha$ is infinite, it fixes a unique vertex $u\in S$.
The deformation space of $S$ is $\alpha$-invariant, so $\alpha(G_u)$ fixes a vertex $u'\in S$.
Since $\Fix\alpha\subset \alpha(G_u)$ fixes only $u$, we have $u'=u$ and $\alpha(G_u)=G_u$.
We now apply the previous analysis to the restriction of $\alpha$ to $G_u$, which    is hyperbolic relative to $\calp_{|G_u}$ by Lemma \ref{Hrufi}. We get a splitting $\Lambda$ of $G_u$ relative to $\calp_{|G_u}$, and 
we obtain  the desired splitting of $G$ by refining   $S/G$   using $\Lambda$  (see Lemma \ref{extens}).
\end{proof}

\section{Rigid groups have finitely many automorphisms}
 \label{sec_Rips}

The goal of this section is to prove Theorem \ref{thm_sci}. 
Let us first recall its statement.

\begin{UtiliseCompteurs}{thm_sci}  
\begin{thm}  Let $G$ be hyperbolic relative to finitely generated subgroups $\calp=\{P_1,\dots,P_n\}$, 
 with $P_i\ne G$. 
Let $\calh=\{H_1,\dots,H_q\}$ be another family of finitely generated subgroups. If $\Out(G;\calp,\mk\calh)$ is infinite, then $G$ splits over an elementary subgroup relative to $\calp\cup\calh$.
\end{thm}
\end{UtiliseCompteurs}

 As mentioned earlier, the proof uses \Rt s. All actions on \Rt s considered here are by isometries. An arc is a subset isometric to an interval $[a,b]\inc \R$ with $a\ne b$. 
 As in the simplicial case, an action on an \Rt\ $T$ is \emph{relative} to   subgroups $H_i$ if each $H_i$ is elliptic (fixes a point) in $T$. 

Because the parabolic groups are not assumed to be slender, we
will need to analyze actions on $\bbR$-trees which are not quite stable.

\subsection{Constructing an \Rt}

\begin{thm}[\cite{BeSz_endomorphisms}]\label{th_BS}  Let $G,\calp,\calh$ be as in Theorem \ref{thm_sci}. 
 If $\Out(G;\calp,\mk\calh)$ is infinite, then $G$ has a non-trivial action on an $\bbR$-tree $T$ relative to $\calp\cup\calh$ such that arc stabilizers are elementary. 
\end{thm}

The proof is essentially in \cite{BeSz_endomorphisms}, noting that a locally elementary subgroup is elementary    by Lemma \ref{borne}.
We also add the remark that  the groups $P_i,H_j$ are elliptic  in $T$. 

\begin{proof} 
Let $\varphi_k$ be automorphisms representing distinct elements of $\Out(G;\calp,\mk\calh)$.
  Let $X $ be  a  $\delta$-hyperbolic space on which $G$ acts as in  Subsection \ref{gene}.
Consider a finite generating set $S$ of $G$, and the minimal displacement  
$d_k=\inf_{x\in X}\max_{s\in S}d_X(x,\phi_k(s).x)$.
Choose  a point $x_k\in X$ where $\max_{s\in S}d_X(x_k,\phi_k(s).x_k)\leq d_k+\frac 1k$. 

Using the Bestvina-Paulin method, it is shown in \cite{BeSz_endomorphisms} that $d_k$ goes to infinity, the   rescaled pointed metric spaces $X_k =(\frac{1}{d_k}X,x_k)$ converge to an \Rt\ $T$ (after taking a subsequence), and the action of $G$ on $X_k$ twisted by $\varphi_k$ converges to a non-trivial isometric action of $G$ on $T$ with locally elementary (hence elementary) 
 arc stabilizers.

We now prove that the action is   relative to $\calp\cup\calh$. Since groups in $\calp\cup\calh$ are finitely generated, it suffices to show that any element $g$ belonging to $P_i$ or $H_j$ is elliptic in $T$. 

Suppose $g $ acts hyperbolically in $T$. Then    there exists $a\in T$ such that $d_T(a,g^2a)=2d_T(a,ga)>0$.
If $a_k$ is an approximation point of $a$ in  $X_k =\frac{1}{d_k}X$, then 
$$\displaystyle \frac{d_X(a_k,\phi_k(g^2)a_k)- d_X(a_k,\phi_k(g)a_k)} {d_k}$$ converges to $d_T(a,g^2a)-d_T(a,ga)=d_T(a,ga)>0$,
so for $k$ large enough  $$d_X(a_k,\phi_k(g^2)a_k)- d_X(a_k,\phi_k(g)a_k)>  2\delta+ \frac{d_k}{2}d_T(a,ga).     
$$
Lemme 9.2.2 of \cite{CDP} implies that $\phi_k(g)$ acts loxodromically  on $X$, with translation length going to infinity since $d_k\to\infty$. This is a contradiction if $g\in P_i$, since   every $\phi_k(g)$ is parabolic in this case. 
If $g\in H_j$, all elements $\varphi_k(g)$ are conjugate so have the same translation length in $X$, also a contradiction.  
\end{proof}

\begin{rem} One can show that a group $H\in\calh $ is elliptic in $T$, even if it is not    assumed to be  finitely generated. We know that every $h\in H$ is elliptic. If $H$ is not elliptic, it fixes an end of $T$, so every finitely generated subgroup of $H$ fixes a ray.  This implies that finitely generated subgroups of $H$ are parabolic, so $H $ is parabolic and therefore elliptic in $T$. 

 On the other hand, Theorem \ref{bf} below requires finite generation.
\end{rem}

\begin{rem}  \label{lent} The hypothesis that automorphisms act trivially on $H_j$ may be weakened. It is sufficient to assume that their  growth under iteration is   slower on $H_j$ than on $G$. 
\end{rem}

\subsection{Hypostability}

 To deduce a splitting as in Theorem \ref{thm_sci}  from the action on the $\bbR$-tree of Theorem \ref{th_BS}, 
we    will generalize 
 the following basic fact (see Theorem \ref{bf2}):

\begin{thm}[{\cite[Thm 9.6]{BF_stable}}] \label{bf}
  Let $G$ be a finitely presented group, and let  $\calq$ 
  be a finite family of finitely generated subgroups.
Assume that $G$ has a non-trivial  stable action on an $\bbR$-tree $T$ relative to $\calq$.
Then $G$ splits relative to $\calq$ over a group $K$ which is an extension $1\ra A\ra K \ra \bbZ^k\ra 1$, where $A$ fixes an arc   of $T$ and $k\geq 0$.
\end{thm}

  Recall that an arc $J$ is stable if any subarc of $J$ has the same stabilizer as $J$. An action is \emph{stable} if every arc $I$ contains a stable subarc $J$.

\begin{cor} \label{po}
   Theorem \ref{thm_sci} holds if every $P_i$ is slender and $G$ is finitely presented.
\end{cor}

\begin{proof}    Assume that every $P_i$ is slender. In this case a subgroup  of $G$ is elementary if and only if it is  slender. In particular, elementary subgroups satisfy   the ascending chain condition, so the action on the $\bbR$-tree
$T$ provided by  Theorem \ref{th_BS} is   stable.
  If furthermore $G$ is finitely presented,  Theorem \ref{bf} (applied with $\calq=\calp\cup\calh$) gives a splitting
that satisfies the conclusion of Theorem \ref{thm_sci}    (note that  $K$ is slender because    $A$ is). 
\end{proof}

In general, however, $G$ is only finitely presented relative to $\calp$, and the action on $T$ only satisfies a weaker property  than stability, which we call hypostability
  (see \cite{Kapovich_sequences} for a different property called semistability).

\begin{dfn}\label{hs}
Let $G$ be a group acting on an $\bbR$-tree $T$.
  The action is \emph{hypostable}
if, for each arc $I\subset T$, there exists a subarc $J\subset I$
satisfying the following {hypostability condition}: if $g\in G$   acts hyperbolically in $T$ and  $gJ\cap J$
is an arc, then $\Stab J=\Stab (gJ)$ (equivalently, $g$ normalizes $\Stab J$).
\end{dfn}

Hypostability is weaker than stability, because any stable arc $J$ satisfies the hypostability condition: if $gJ\cap J$ is an arc,  $\grp{\Stab(J),\Stab(gJ)}$ is contained in
the stabilizer of $gJ\cap J$, which coincides with  $\Stab(J)$ and $\Stab(gJ)$ 
by stability of $J$ and $gJ$.

\begin{lem}\label{lem_elem_hypostable} Let $G$ be hyperbolic relative to finitely generated subgroups $\calp=\{P_1,\dots,P_n\}$. Any action of $G$ on an \Rt\ $T$ 
  relative to $\calp$ with elementary arc stabilizers 
is hypostable.
\end{lem}

\begin{proof}    Let $C$ be such that any elementary subgroup $H$ of $G$ of cardinality $>C$ is contained in a unique maximal elementary subgroup $E(H)$ (see Lemma \ref{borne}). 
Let $I\subset T$ be an arc. If the stabilizer of every subarc has cardinality at most $C$, then 
a subarc $J\subset I$ whose stabilizer has the greatest cardinality is stable and we are done.

Otherwise, consider $J\subset I$ whose stabilizer $H=\Stab(J)$ has cardinality $>C$.
The stabilizer of every subarc of $J$ is elementary, so is contained in $E(H)$.
If subgroups of $E(H)$ satisfy the ascending chain condition, in particular if  $E(H)$ is  virtually cyclic, then $J$ contains a stable subarc.
Thus we can assume that $E(H)$ is parabolic. 

We prove hypostability by showing that
any $g$ such that $gJ\cap J$ contains an arc is elliptic in $T$. Indeed, $\grp{H,H^g}$   fixes an arc in $T$, so is elementary.
It follows that $\grp{H,H^g}\subset E(H)$, so $E(H)=E(H^g)=E(H)^g$.
Since $E(H)$ is its own normalizer, we get $g\in E(H)$. But $T$ is relative to $\calp$, so  $g$ is elliptic.
\end{proof}

\begin{example}  
  We sketch the construction of   an   action as in  
   Lemma \ref{lem_elem_hypostable} which is not stable.  
Let $G$ be
the free product $G=P*\Z$, with   $P$ a  (non-slender) finitely generated group
containing a copy of the free abelian group on a countable basis $\bbZ^{(\bbQ)}$.  Note that $G$ is hyperbolic relative to $\{P,\Z\}$. Informally, identifying the edge of the free product with $[0,1]$, 
one can produce an \Rt\   from the Bass-Serre tree of this splitting
by folding the group  $\bbZ^{([0,\frac pq]\cap\bbQ)}$ on a length $\frac pq$ for all $0<\frac pq<1$.
The stabilizer of an arc $[a,b]\subset [0,1]$ is then $\bbZ^{([0,b]\cap\bbQ)}$. This action is hypostable but unstable.
\end{example}

\begin{thm} \label{bf2} 
Theorem \ref{bf} holds if the action on the $\bbR$-tree is only assumed to be hypostable,
and the group $G$ is only assumed to be finitely presented relative to $\calq$.
\end{thm}

 The proof will be given in the next subsection.
The following corollary is an immediate consequence of Lemma \ref{lem_elem_hypostable} and Theorem \ref{bf2}.

\begin{cor}\label{cor_RipsRH}
Let $G$ be a relatively hyperbolic group,
 with $\calp,\calh$   as in Theorem \ref{thm_sci}. If $G$ acts non-trivially on an \Rt\ $T$ 
  relative to  $\calp\cup\calh$ with elementary arc stabilizers,
then $G$ splits  over an elementary subgroup relative to $\calp\cup\calh$.  \qed
\end{cor}

Theorem \ref{thm_sci} follows immediately 
from this corollary, using the $\bbR$-tree provided by Theorem \ref{th_BS}. 
A refinement of Corollary \ref{cor_RipsRH} will be given in Subsection \ref{zm}.

\begin{rem} Let $G$ and $T$ be as in   Theorem \ref{bf}. If $T$ is not a line, one can get a splitting over a group $K$ which is an  extension of $\bbZ$ or $\bbZ/2\bbZ$ by
a group $A$ fixing an arc in $T$  (see \cite{BF_stable}). 
One can also approximate $T$ (in the equivariant Gromov topology) by simplicial trees with  controlled edge stabilizers,  as in  \cite{Gui_approximation}. The same facts are true under the assumptions of Theorem \ref{bf2}.
\end{rem}

\subsection{Proof of Theorem \ref{bf2}}  

Recall that a subtree $Y\subset T$ is \emph{indecomposable}  \cite{Gui_actions} if, given  arcs $I,L\subset Y$,  there exist  $g_1,\dots,g_n\in G$ such that $L\subset g_1I\cup\dots\cup g_nI$,
and every $g_iI\cap g_{i+1}I$ is an arc. We call $g_1,\dots,g_n $ an \emph{$I$-covering} of $L$. 
 
\begin{lem}\label{lem_superpos}
  Assume that $Y\subset T$ is an indecomposable subtree.
Given 
  two arcs $I,L\subset Y$ with $I\subset L$,
there exists an $I$-covering $ g_1,\dots,g_r $ of $L$ such that $g_1=1$ and 
every   $ g_{i}g_{i-1}\m$  is hyperbolic in $T$.
\end{lem}

\begin{proof}

Given any $I$-covering of $L$, there exists $i$ such that $I\cap g_i I$ is an arc, and therefore $1,g_i,g_{i-1}, \dots, g_1,g_2,\dots, g_n$ is an $I$-covering starting with 1. From now on, we only consider coverings starting with 1. 
The interval covered will always be $L$. 

We fix an orientation of $I$. It induces an orientation of $gI$ for any $g\in G$. If $I\cap gI$ is an arc, the orientations of $I$ and $gI$ may agree or disagree on   this arc. 
  We first claim that there exists     an  $I$-covering $ 1=a_1,\dots,a_p $ of $L$ 
  such that, for each $i$, 
 the orientation of $a_iI$ agrees with that  of $a_{i+1}I$ on their
intersection (we say that such an $I$-covering is orientation-preserving).

If not, we can find arcs $gI,g'I$ whose intersection is an arc on which the orientations disagree. 
Define $g_0=g'^{-1}g$ and $J=I\cap g_0 I$;   the orientations of $I$ and $g_0I$ disagree on $J$.
Now consider    a $J$-covering $1=b_1,\dots,b_p $ of $L$. 
It is also an $I$-covering   of $L$. Since  $J= I\cap g_0 I$, we get another $I$-covering of $L$
if we replace some of the $b_i$'s by $b_ig_0$. Starting with $a_1=b_1=1$, we then define $a_i$ inductively as either $b_i$ or $b_ig_0$, 
making sure that orientations agree. This proves the claim.

 Next note that  there exists a hyperbolic element $h\in G$   mapping an arc $J'\subset I$ to a different arc $h(J')\subset I$ in an orientation-preserving way. To see this, first choose $h_1$ mapping an arc $J_1\inc I$ to a disjoint arc $J_2\inc I$. If orientation is reversed, choose $h_2$ mapping an arc $J_3\inc J_2$ to an arc $J_4\inc I$ different from  $ J_1$ and $J_3$. Then take $h$ equal to $h_2$ or $h_2h_1$.

 We can now conclude.
Let $1=a_1,\dots,a_r$ be an orientation-preserving $J'$-covering   of $L$. 
Since $J'\inc I\cap h\m I$, 
we get another orientation-preserving $I$-covering if we replace some of the $a_i$'s by $a_ih\m$.
If  $a_{i}a_{i-1}\m $ is not hyperbolic,  it is the identity on  $a_{i-1}J'\cap a_{i} J'$.
We therefore get 
the required $I$-covering of $L$ by defining $g_i$ inductively as $a_i$ or $a_ih\m$ so that $g_{i}g_{i-1}\m $ is not the identity on  $a_{i-1}J'\cap a_{i} J'$.
\end{proof}

\begin{cor}\label{cor_hypostable}
Let $T$ be hypostable, and let $Y\subset T$ be an indecomposable subtree.
Any element $g\in G$ fixing an arc in $Y$ fixes the whole of $Y$.
In particular, any arc in $Y$ is stable.
\end{cor}

\begin{proof}
 Assume that $g$ fixes an arc $I\subset Y$.  Given $x\in Y$, we aim to prove that $g$ fixes $x$. 
After making $I$ smaller, we can assume that there is an arc $L$ containing $x$ and $I$.
Definition \ref{hs} provides a subarc $J\inc I$ satisfying the hypostability condition.   Consider a $J$-covering
$1= g_1,\dots,g_r $ of $L$ as in 
Lemma \ref{lem_superpos}.    Since  $g_{i+1}g_i\m $, hence also $g_i\m g_{i+1}$, is hyperbolic, hypostability of $J$ implies that all arcs $g_iJ$ have the same stabilizer. The element $g$ fixes $g_1J=J$, so it fixes every $g_iJ$ and therefore $L$. In particular, $g$ fixes $x$.
\end{proof}

\begin{proof}[Proof of Theorem \ref{bf2}]  
  We explain how to adapt the arguments in \cite{BF_stable,Gui_approximation}.

  Let $\calq=\{Q_1,\dots,Q_q\}$.
Let $ \grp{S_i\mid\calr_i}$ be a presentation of   $Q_i$, with
$S_i$ a finite generating set and $\calr_i$ a possibly infinite set of relators.
Let $\grp{S\mid\calr}$ be a presentation of $G$ such that $S$ is a finite generating set of $G$ containing each $S_i$,
and $\calr$ is the union of $\calr_1\cup\dots\cup\calr_q$   with finitely many additional relators.

Consider a finite subtree $K\subset T$, \ie the convex hull of finitely many points.  
We explain how to choose $K$ large enough so   as to yield  a resolution of $T$
as in \cite[Definition 2.2]{Gui_approximation}, even though $G$ is only relatively finitely presented.

For each $s\in S$, we consider $K_s=K\cap s\m K$ and the restriction  $\phi_s:K_s \ra sK_s$ of $s$ (we may assume that no $K_s$ is empty). We then define the suspension $\Sigma$ as the foliated $2$-complex obtained by gluing 
foliated bands $K_s\times [0,1]$ to $K$, where we glue $(x,0)$ to $x$ and $(x,1)$ to $\phi_s(x)$.
Note that $\pi_1(\Sigma)$ is naturally
identified with the free group on $S$.

Next,  we  need all relators of $\calr$ to be represented by
loops contained in leaves of $\Sigma$. Since  each   $Q_i$ fixes a point $p_i$ in $T$, and $S$ contains $S_i$,   requiring that  
 $K$ contains $p_1,\dots,p_q$ takes care of $\calr_1\cup\dots\cup\calr_q$. 
There remain    finitely many other relators,
and, as in \cite{BF_stable,Gui_approximation}, one can choose $K$   so that
they   also are  represented by loops contained in leaves.

The complex $\Sigma$ provides  a resolution of $T$ in the sense of Definition 2.2 of \cite{Gui_approximation}
(as pointed out in  \cite{Gui_approximation}, the set $\calc$ of curves contained in leaves mentioned in Definition 2.2 is not assumed to be finite).
 Obtaining a resolution  is the only place where finite presentation is used in \cite{Gui_approximation}.

As for stability, it  is used only in Proposition 4.3 of \cite{Gui_approximation} 
to prove that, if an element fixes an arc 
in the subtree $T_{\Gamma_v}\subset T$ corresponding to a
minimal component of $\Sigma$, then it fixes the whole of $T_{\Gamma_v}$.
By \cite[Proposition 1.25]{Gui_actions}, the geometric $\bbR$-tree dual to a minimal component of $\Sigma$ is indecomposable,
and by Lemma 1.19(1) of \cite{Gui_actions} its image $T_{\Gamma_v}\subset T$
is an indecomposable subtree of $T$.
 Corollary \ref{cor_hypostable} then replaces Proposition 4.3 of \cite{Gui_approximation}
 under our hypostability assumption.

The rest of the argument of \cite{Gui_approximation} applies without modification.
As in Proposition 4.1 of \cite{Gui_approximation},
the tree dual to $\Sigma$ is a graph of actions on $\bbR$-trees $T(\calg')$,
such that  arc stabilizer of the vertex actions lie in the kernel 
of these actions. Applying Propositions 5.2, 7.2 and 8.1 of \cite{Gui_approximation}, 
one can replace these vertex actions by actions on simplicial trees 
whose stabilizers are abelian modulo the kernel. 
By Bass-Serre theory, this provides a splitting of $G$ over the extension
of an abelian group $A$ by the kernel $K$ of a vertex action. 
This splitting is relative to $\calq$ as in the Reduction Lemma in \cite[\S 4]{Gui_approximation}.
\end{proof}

\subsection{$\Zmax$ splittings} \label{zm}

Say that a subgroup of a relatively hyperbolic group is $\Zmax$
if it is maximal for inclusion among non-parabolic virtually cyclic subgroups with infinite center. 

\begin{thm}\label{thm_RipsZmax}
Let $G$ be hyperbolic relative to $\calp=\{P_1,\dots,P_n\}$, with $P_i$  finitely generated. 
Let $\calh=\{H_1,\dots H_q\}$ be a (possibly empty)   family of finitely generated subgroups.
Assume that $G$ acts non-trivially on an \Rt\ $T$ 
  relative to $\calp\cup\calh$, and that arc stabilizers are either finite, parabolic or $\Zmax$.

Then $G$ splits relative to $\calp\cup\calh$ over a finite, parabolic or $\Zmax$ subgroup.
\end{thm}

   Corollary \ref{cor_RipsRH} provides a splitting over an elementary group $A$. Here we assume that every loxodromic arc stabilizer of $T$ is $\Zmax$, and we   claim that the same is true for $A$: if it is loxodromic, then it is $\Zmax$. 

\begin{proof} 
$\bullet$   Assume first that the foliated $2$-complex $\Sigma$  constructed in the proof of Theorem \ref{bf2}  
has a minimal component $\Sigma_v$. Let  $G_v$ be the image of its fundamental group in $G$,
and let $T_{G_v}\subset T$ be the corresponding subtree of $T$.
In particular, $G_v$ is not elliptic in $T$.
By \cite[Theorem 5.13]{BF_stable} or \cite[Theorem 3.1]{Gui_approximation},
$G_v$ is a vertex group in a decomposition of $G$ as a graph of groups $\Gamma$  relative to $\calp\cup\calh$.
 All arcs in $T_{G_v}$ have the same stabilizer  $F$, 
 a normal
subgroup of  $G_v$. 

We claim that $F$ is finite. Otherwise,  there are two cases. 
If $F$ is   non-parabolic, hence  virtually cyclic, 
it   has finite index in its normalizer and therefore in $G_v$, so $G_v$ is elliptic, a contradiction.
If $F$ is infinite and contained in a maximal parabolic group $P$, then 
almost malnormality of $P$ implies that the normalizer of $F$ is contained in $P$,
so $G_v\subset P$. Since $P$ is elliptic in $T$, this is also a contradiction.
Thus $F$ is finite. 

 We now distinguish several cases, depending on the nature of the minimal component $\Sigma_v$.

First,  $\Sigma_v$ cannot be a homogeneous  (axial,  toral) component since $G_v$ would then be
virtually $\bbZ^k$ for some $k\geq 2$ (\cite[Theorem 9.4(2)]{BF_stable} or \cite[section 5.1]{Gui_approximation}), 
hence parabolic, contradicting  ellipticity of   parabolic groups  in $T$.

If $\Sigma_v$ is an exotic  (Levitt, thin) minimal component, 
one obtains a splitting of $G$  over $F$, and we are done (\cite[Proposition 7.2]{Gui_approximation}, \cite[Theorem 9.4(3)]{BF_stable}).

If $\Sigma_v$ is a surface  (IET) component,
then by \cite[Theorem 9.4(1)]{BF_stable} or \cite[section 8]{Gui_approximation}, 
after performing some moves, one can assume that $\Sigma_v$ is a surface with boundary, and $G_v/F$
is the fundamental group of a $2$-orbifold with conical singularities supporting a measured foliation with dense leaves.
Moreover, 
 $G_v$ is a QH vertex group (with fiber $F$) of $\Gamma$. 

Let $A\subset  G_v$ be the preimage of the fundamental group of 
an essential two-sided simple closed curve not bounding a M\"obius band, and not boundary parallel.  Then $G$ splits over $A$ 
relative to  $\calp\cup\calh$. We check that $A$ is $\Zmax$.
Clearly, $A$ is virtually cyclic with infinite center, and is maximal among virtually cyclic subgroups of $G_v$. 
Since $G_v$ is QH and $A$ is not conjugate into a boundary subgroup, it easily follows that
$A$ is $\Zmax$ in $G$.

$\bullet$ 
The remaining case is  when  $\Sigma$ has no minimal component.
In this case, all leaves of $\Sigma$ are finite, and the dual tree $T_\Sigma$ is simplicial.
Its edge stabilizers fix an arc in $T$, so we are done if one of these edge stabilizers
is finite or parabolic.
Otherwise, arc stabilizers of $T_\Sigma$ are virtually cyclic with infinite center but may fail to be $\Zmax$.
 If this happens, we have to enlarge  the   finite tree $K$ used to construct $\Sigma$. 

By \cite{LP}, we can find an exhaustion of $T$ by an increasing sequence of finite subtrees $K_k$
such that the corresponding dual trees $T_{\Sigma_k}$ strongly converge to $T$. We refer to \cite{LP} for
the definition of strong convergence; we will only use the fact that, if $A$ is a finitely generated
group fixing an arc in $T$, then $A$ fixes an edge in $T_{\Sigma_k}$ for large enough $k$.

We can assume that all dual trees $T_{\Sigma_k}$ are simplicial, and that their edge stabilizers are
infinite and not parabolic. Let $A_0$ be an edge stabilizer of $T_{\Sigma_0}$.
Then $A_0$ fixes an arc $I$ in $T$.  By the hypothesis on arc stabilizers of $T$, the stabilizer
of $I$ is a $\Zmax$ subgroup $A\supset A_0$,
which fixes an edge $e$ of $T_{\Sigma_k}$ for $k$ large enough.   Since $G_e$ contains $ A$ and fixes an arc in $T$, it is equal to $A$, so 
$T_{\Sigma_k}$ provides the desired splitting.
\end{proof}

A similar proof yields the following results.

\begin{thm}\label{thm_Rips_parab}
Let $G$ be hyperbolic relative to $\calp=\{P_1,\dots,P_n\}$,  with $P_i$ slender,
and let $\calh=\{H_1,\dots H_q\}$ be a    family of finitely generated subgroups.
 Assume that 
 $G$ acts non-trivially on an \Rt\ $T$ 
 relative to $\calp\cup\calh$, with elementary arc stabilizers.

Then $G$ splits relative to $\calp\cup\calh$  over a $\Zmax$ subgroup or  over the stabilizer of an arc of $T$.  \qed
\end{thm}

\begin{cor} Let $G$ be a  toral 
  relatively hyperbolic group. Consider  a non-trivial action of $G$ 
on an $\bbR$-tree  relative to non-cyclic abelian subgroups. If arc stabilizers are abelian  and stable under taking roots, then $G$ splits (relative to non-cyclic abelian subgroups) over an abelian subgroup  stable under taking roots.  \qed
\end{cor}

\begin{flushleft}
Vincent Guirardel\\
Institut de Recherche Math\'ematique de Rennes\\
Universit\'e de Rennes 1 et CNRS (UMR 6625)\\
263 avenue du G\'en\'eral Leclerc, CS 74205\\
F-35042  RENNES C\'edex\\
\emph{e-mail:}\texttt{vincent.guirardel@univ-rennes1.fr}\\[8mm]

Gilbert Levitt\\
Laboratoire de Math\'ematiques Nicolas Oresme\\
Universit\'e de Caen et CNRS (UMR 6139)\\
BP 5186\\
CS 14032\\
14032 CAEN cedex 5\\
France\\
\emph{e-mail:}\texttt{levitt@unicaen.fr}\\

\end{flushleft}

\end{document}